\documentclass[11 pt]{amsart}

\usepackage{amsmath,amsthm,amssymb,amscd,graphicx,psfrag}
\usepackage[colorlinks=true, pdfstartview=FitV, linkcolor=blue, citecolor=red, urlcolor=blue]{hyperref}

\newtheorem{thm}{Theorem}[section]
\newtheorem{cor}[thm]{Corollary}
\newtheorem{lem}[thm]{Lemma}

\newtheorem{prop}[thm]{Proposition}

\theoremstyle{definition}

\theoremstyle{remark}
\newtheorem{rem}[thm]{Remark}

\numberwithin{equation}{section}

\def\be#1 {\begin{equation} \label{#1}}
\newcommand{\ee}{\end{equation}}

\newcommand{\mb}{\medskip\noindent}

\newcommand{\R}{\mathbb R}

\def \vsp {\vspace{6pt}}

\begin{document}

\author{Fr\'ed\'eric Bernicot}
\address{Fr\'ed\'eric Bernicot - CNRS - Universit\'e Lille 1 \\ Laboratoire de math\'ematiques Paul Painlev\'e \\ 59655 Villeneuve d'Ascq Cedex, France}
\curraddr{}
\email{frederic.bernicot@math.univ-lille1.fr}

\author{Pierre Germain}
\address{Pierre Germain - Courant Institute of Mathematical Sciences \\ New York University \\ 251 Mercer Street \\ New York, N.Y. 10012-1185 \\ USA}
\email{pgermain@cims.nyu.edu}

\title[Bilinear dispersive estimates]{Bilinear dispersive estimates via space-time resonances. Part I : the one dimensional case.}



\date{\today}

\begin{abstract}  We prove new bilinear dispersive estimates. They are obtained and described via a bilinear time-frequency analysis following the space-time 
resonances method, introduced by Masmoudi, Shatah, and the second  author.
They allow us to understand the large time behavior of solutions of quadratic dispersive equations.
\end{abstract}

\maketitle

\begin{quote}
\footnotesize\tableofcontents
\end{quote}

\section{Introduction}

\subsection{Linear dispersive and Strichartz estimates} A linear, hyperbolic equation is called dispersive if the group velocity of a wave packet depends on its frequency. 
In order to remain concise, we only discuss in this paragraph the Schr\"odinger equation\
$$ \left\{ \begin{array}{l}
 \partial_t u -i\Delta u =0 \\
 u_{|t=0} = f,
\end{array} \right.$$
whose solution we denote $u(t) = e^{it\Delta} f$. It is the prototype of a dispersive equation. A first way to quantify dispersion is provided by 
the so-called ``dispersive estimates'' which read, in the case of the linear Schr\"odinger equation
$$ \|e^{it\Delta} f\|_{L^p(\R^d)} \lesssim t^{\frac{d}{p}-\frac{d}{2}} \|f\|_{L^{p'}(\R^d)}\qquad \mbox{if $2 \leq p \leq \infty$}.$$
Another way of quantifying dispersion is provided by Strichartz estimates, which first appeared in the work of Strichartz \cite{Strichartz} (then extended by Ginibre and Velo 
in \cite{GV} and the end-points are due to Keel and Tao in \cite{KT}). They read
$$ \|e^{it\Delta} f\|_{L^p L^q(\R^+ \times \R^d)} \lesssim \|f\|_{L^2(\R^d)}$$
for every admissible exponents $(p,q)$ which means: $2\leq p,q \leq \infty$, $(p,q,d)\neq (2,\infty,2)$ and
$$\frac{2}{p}+\frac{d}{q}=\frac{d}{2}.$$

Let us just point out the situation if the Euclidean space $\R^d$ is replaced by a compact Riemannian manifold. In that case, any constant function is solution of the 
free Schr\"odinger equation and therefore the dispersive estimate fails for large $t$. Even more, it also fails locally in time. Then Schrichartz estimates may only hold with 
a finite time scale and a loss of derivatives (the data $f$ is controled in a positive order Sobolev space), which was obtained by Bourgain for the torus in 
\cite{Bourgain,Bourgain2} and then extended to general manifolds by Burq, G\'erard and Tzvetkov in \cite{BGT}.

\subsection{Bilinear Strichartz estimates}

Recently bilinear (and more generally multilinear) analogs of such inequalities have appeared. They correspond to controlling the size of the (pointwise) product of two linear solutions, for instance:
\begin{equation}
\label{parrot1}
\left\| vw \right\|_{L^p L^q(\R^+ \times \R^d)} \lesssim \|f\|_{L^2(\R^d)}\|g\|_{L^2(\R^d)} \;\;\;\;\mbox{with}\;\;\;\;\left\{ \begin{array}{l} i \partial_t v + \Delta v = 0, \;\;\;v(t=0)=f  \vsp \\ i \partial_t w + \Delta w = 0, \;\;\;w(t=0) = g \vsp \end{array} \right.
\end{equation}
or the solution to the inhomogeneous linear problem, the right hand side being given by the product of two linear solutions:
\begin{equation}
\label{parrot2}
\left\| u \right\|_{L^p L^q(\R^+ \times \R^d)}  \lesssim \|f\|_{L^2(\R^d)}\|g\|_{L^2(\R^d)} \;\;\;\;\mbox{with}\;\;\;\;\left\{ \begin{array}{l} i \partial_t v + \Delta v = 0, \;\;\;v(t=0)=f \vsp \\ i \partial_t w + \Delta w = 0, \;\;\;w(t=0) = g \vsp \\ i \partial_t u + \Delta u = vw, \;\;\;u(t=0) = 0. \vsp \end{array} \right.
\end{equation}
In a first line of research, $p=q=2$, which is related to the use of $X^{s,b}$ spaces in order to solve nonlinear dipsersive equations. We would like to cite in this respect 
Bourgain \cite{Bourgain} and Tao~\cite{Tao}. If the Euclidean space is replaced by a manifold, we refer to Burq, G\'erard and Tzvetkov \cite{BGT2} and Hani~\cite{Hani}.
The case of the wave equation is treated in Klainerman, Machedon, Bourgain and Tataru \cite{KMBT}, and Foschi and Klainerman~\cite{KF}. In all these works, $f$ and $g$
are chosen with vastly different frequency supports, and the focus is on understanding the effect on the implicit constant.

Another line of research considers the case where $p$ and $q$ are different from $2$: see Wolff~\cite{Wolff} for the case of the wave equation and Tao~\cite{Tao2} for Schr\"odinger
equation. The problem becomes then related to deep harmonic analysis questions (the restriction conjecture), and the optimal estimates are not known in high dimension.

In the present note, our goal is different from the two directions which have been mentioned: we aim at finding a decay rate in time (rather than integrability properties),
and at understanding the effect of localized data.

\subsection{The set up}

From now on, the dimension $d$ of the ambiant space is set equal to 1. Let $a,b,c$ be smooth real-valued functions on $\R$, and fix a 
smooth, compactly supported bilinear symbol $m$ on the frequency plane $\R^2$. 
We denote $T_m$ the associated pseudo-product operator (a precise definition of $T_m$ is given in Section~\ref{notations}; $T_m$ can be thought of as a 
generalized product operator, and our setting of course includes classical products between functions which are compactly supported in Fourier space).
Consider then the equation
\begin{equation}
\label{eq:dispersive}
\left\{ \begin{array}{l} i\partial_t u + a(D) u = T_{m}\left( v , w \right) \vsp \\
i\partial_t v + b(D) v = 0 \vsp \\
i\partial_t w + c(D) w = 0 \vsp 
\end{array} \right.
\quad \mbox{with} \quad
\left\{ \begin{array}{l} u(t=0) = 0 \vsp \\ v(t=0) = f \vsp \\ w (t=0) = g. \vsp \end{array} \right.
\end{equation}
The unknown functions are complex valued, and this system is set in the whole space: $f$ and $g$ map $\mathbb{R}$ to $\mathbb{C}$, whereas $u$, $v$, and $w$ map $\mathbb{R}^2$
to $\mathbb{C}$. The above system is meant to understand the nonlinear interaction of free waves, which is of course the first step towards 
understanding a nonlinear problem.

\bigskip
We will most of the time, but not always, assume
\begin{equation}
(H) \quad \quad \mbox{The second derivatives $a''$, $b''$, $c''$ are bounded away from zero}
\end{equation}
Under these hypotheses, it is well known that the groups $e^{ita(D)}$, $e^{itb(D)}$, $e^{itc(D)}$ satisfy the following estimates (we denote $S(t)$ for any of these groups)
\begin{itemize}
\item $\displaystyle \left\| S(t) f \right\|_{L^{p'}} \lesssim |t|^{\frac{1}{2}-\frac{1}{p}} \left\| f \right\|_{L^p}$ for $p\in [1,2]$ (dispersive estimates).
\item $\displaystyle \left\| S(t) f \right\|_{L^p_t L^q} \lesssim \left\| f \right\|_{L^2}$ if $\frac{2}{p}+\frac{1}{q}=\frac{1}{2}$, and $2 \leq p,q \leq \infty$
(Strichartz estimates).
\end{itemize}

\bigskip

The question we want to answer is: {\bf given $f$ and $g$ in $L^2$ (or weighted $L^2$ spaces), how does $u$ grow, or decay, in $L^p$ spaces, $2\leq p \leq \infty$?}

The answer of course depends on $a$, $b$, $c$, and the crucial notion is that of space-time resonance.

\subsection{Space-time resonances}

Using Duhamel's formula, $u(t,\cdot)$ is given by the following bilinear operator $u(t,\cdot)=T_t(f,g)$ with
$$
T_t(f,g)(x) = \int_0^t \int \int e^{ix(\xi+\eta)} e^{ita(\xi)} e^{is\phi(\xi,\eta)} m(\xi,\eta) \widehat{f}(\eta) \widehat{g}(\xi-\eta)\,d\xi\,d\eta\,ds
$$
or, to put it in a more concise form,
$$ T_t(f,g)\overset{def}{=} - i e^{ita(D)} \int_0^t T_{m e^{is\phi}} (f,g) ds, $$
where 
$$\phi(\xi,\eta)\overset{def}{=} -a(\xi+\eta)+b(\xi)+c(\eta).$$
The goal of this article is thus to understand the behavior for large time $t>>1$ and some exponent $q\in[2,\infty]$ of
$$ \left\| T_t(f,g) \right\|_{L^q}, \qquad f,g\in L^2.$$
We sometimes find it convenient to write $u(t)$ as
$$
u(t) = \mathcal{F}^{-1} \int_0^t\int_\R e^{ita(\xi)} e^{is\Phi(\xi,\eta)} \mu(\xi,\eta) \widehat{f}(\xi-\eta) \widehat{g}(\eta) \,d\eta\,ds
$$
where
$$
\Phi(\xi,\eta) \overset{def}{=} -a(\xi) + b(\xi-\eta) + c(\eta) =\phi(\xi-\eta,\eta)\quad \mbox{and} \quad \mu(\xi,\eta) \overset{def}{=} m(\xi-\eta,\eta).
$$

Viewing this double integral as a stationary phase problem, it becomes clear that the sets where the phase is stationary in $s$, respectively $\eta$,
$$
\Gamma \overset{def}{=} \{(\xi,\eta) \:\mbox{such that}\;\Phi(\xi,\eta) = 0\} \quad \mbox{and} \quad \Delta \overset{def}{=} \{(\xi,\eta) \:\mbox{such that}\; 
\partial_\eta \Phi(\xi,\eta) = 0\}
$$
will play a crucial role. Even more important is their intersection $\Gamma \cap \Delta$.

The sets $\Gamma$ and $\Delta$ are, respectively, the sets of time and space resonances; their intersection is the set of space-time resonant sets.
A general presentation, stressing their relevance to PDE problems, can be found 
in~\cite{Pierre}; instances of applications are~\cite{GMS1}~\cite{GMS2}~\cite{GMS3}~\cite{G}~\cite{GM}. 

In order to answer the question asked in the previous subsection, one has to distinguish between various possible geometries of $\Gamma$ and $\Delta$ (which can
be reduced to a discrete set, or curves, with vanishing curvature or not, etc...); possible orders of vanishing of $\Phi$, respectively $\partial_\eta \Phi$, on $\Gamma$, 
respectively $\Delta$; and type of intersection of $\Gamma$
and $\Delta$ (at a point or on a dimension 1 set, transverse or not, etc...). Considering all the possible configurations would be a daunting task, we will therefore focus
on a few relevant and ``generic'' examples.

\begin{itemize}
\item We shall study the influence of time resonances alone, without caring about space resonances: in other words, we shall study various configurations for $\Gamma$, without making 
any assumptions on $\Delta$. This amounts essentially to considering the worst possible case as far as $\Delta$ is concerned.
\item Similarly, we shall study the influence of space resonances alone, without caring about time resonances.
\item When putting space and time resonances together, we will assume a ``generic'' configuration: $\Gamma$ and $\Delta$ are smooth curves, and they intersect transversally
at a point. Aside of being generic, this configuration has a key importance for many nonlinear PDE; this will be explained in the next subsection.
\end{itemize}

\subsection{Space-time resonant set reduced to a point} 
As was just mentioned, the case where $\Gamma$ and $\Delta$ are curves which intersect transversally at a point will be examined carefully in this article. It is of course 
the generic situation, but it also occurs in a number of important models from physics; we would like to give a few examples here. We of course restrict the discussion to the case 
of one-dimensional models.

For simple equations of the form $i \partial_t u + \tau(D) u = Q(u,\bar u)$, where $u$ is scalar-valued, $Q$ quadratic (ie, we retain only the quadratic part of the nonlinearity), 
and $\tau(\xi) = |\xi|^\alpha$ is homogeneous, the space-time resonant set of the various possible interactions between $u$ and $\bar u$ is never reduced to a point. This is the 
case for standard equations such as NLS, KdV, wave equations... 

However, if $\tau$ is not supposed to be homogeneous any more, the space-time resonant set might be reduced to a point. This
is in particular the case for the water wave equation (ideal fluid with a free surface) in the following setting: close to the equilibrium given by a flat surface and zero 
velocity, including the effects of gravity $g$ and capillarity $c$, with a constant depth $d$ (perhaps infinite). The dispersion relation for the linearized problem is then given 
by
$$
\tau(\xi) = \tanh(d|\xi|)\sqrt{g|\xi|+c|\xi|^3}.
$$

For more complex models, $u$ is vector-valued, and the system accounts for the interaction of waves with different dispersion relations. It is then often the case that the space-time
resonance set is reduced to a point. We mention in particular
\begin{itemize}
\item The Euler-Maxwell system, describing the interaction of a charged fluid with an electromagnetic field (see~\cite{GM} for a mathematical treatment of this equation dealing 
with space-time resonances). Many other models of plasma physics could also be mentioned here.
\item Systems where wave and (generalized) Schr\"odinger equations are coupled: for instance the Davey-Stewartson, Ishimori, Maxwell-Schr\"odinger, Zakharov systems.
\end{itemize}

\subsection{Organization of the article}

The present article is organized as follows
\begin{itemize}
\item Asymptotic equivalents for $u$ are derived in \textit{Section~\ref{ae}}, for $f$ and $g$ smooth and localized. 
Three cases are considered: $\Gamma=\emptyset$, $\Delta=\emptyset$, and $\Gamma$ and $\Delta$ are curves intersecting
transversally at a point.
\item In \textit{Section~\ref{nld}}, we establish estimates on $u$ if $f$ and $g$ belong to $L^2$; we shall then only rely on time resonances.
\item In \textit{Section~\ref{ld}}, we establish estimates on $u$ if $f$ and $g$ belong to weighted $L^2$ spaces. We consider in particular the case when the space-time resonant set
is reduced to a point.
\item We detail in \textit{Appendix~\ref{Ame}} some results on boundedness of multilinear operators.
\item Finally, one-dimensional oscillatory integrals are studied in \textit{Appendix~\ref{Aodoi}}.
\end{itemize}

\subsection{Notations}
\label{notations}
We adopt the following notations
\begin{itemize}
\item $A \lesssim B$ if $A \leq C B$ for some implicit constant $C$. The value of $C$ may change from line to line.
\item $A \sim B$ means that both $A \lesssim B$ and $B \lesssim A$.
\item If $f$ is a function over $\mathbb{R}^d$ then its Fourier transform, denoted $\widehat{f}$, or $\mathcal{F}(f)$, is given by
$$
\widehat{f}(\xi) = \mathcal{F}f (\xi) = \frac{1}{(2\pi)^{d/2}} \int e^{-ix\xi} f(x) \,dx \;\;\;\;\mbox{thus} \;\;\;\;f(x) = \frac{1}{(2\pi)^{d/2}} \int e^{ix\xi} \widehat{f}(\xi) \,d\xi.
$$
(in the text, we systematically drop the constants such as $\frac{1}{(2 \pi)^{d/2}}$ since they are not relevant).
\item The Fourier multiplier with symbol $m(\xi)$ is defined by
$$
m(D)f = \mathcal{F}^{-1} \left[m \mathcal{F} f \right].
$$
\item The bilinear Fourier multiplier with symbol $m$ is given by
$$ T_m(f,g)(x) \overset{def}{=} \int_{\R^2} e^{ix(\xi+\eta)} \widehat{f}(\xi) \widehat{g}(\eta) m(\xi,\eta)\, d\xi d\eta 
= \mathcal{F}^{-1} \int m(\xi-\eta,\eta) \widehat{f}(\xi-\eta) \widehat{g}(\eta)\,d\eta.$$
\item The japanese bracket $\langle \cdot \rangle$ stands for $\langle x \rangle = \sqrt{1 + x^2}$.
\item The weighted Fourier space $L^{p,s}$ is given by the norm $\|f\|_{L^{p,s}} = \|\langle x \rangle^s f \|_{L^p}$.
\item If $E$ is a set in $\mathbb{R}^d$, then $E_\epsilon$ is the set of points of $\mathbb{R}^d$ which are within $\epsilon$ of a point of $E$.
\end{itemize}

\section{Asymptotic equivalents}

\label{ae}

\subsection{Preliminary discussion}

Our aim in this section is to obtain asymptotic equivalents, as $t \rightarrow \infty$, for the solution $u$ of~(\ref{eq:dispersive}), under the simplifying assumption that $f$ and 
$g$ are very smooth and localized. Hypotheses on $a$, $b$, $c$ are needed, and the variety of possible situations is huge; we try to focus on the most representative, or generic 
situations. First, we will assume in this whole section that $(H)$ holds: this gives decay for the linear waves. For bilinear estimates, everything hinges on  the vanishing 
properties of $\Phi$ and $\partial_\eta \Phi$, where
$$
\Phi(\xi,\eta) = -a(\xi) + b(\xi-\eta) + c(\eta).
$$
We will distinguish three situations, which correspond to the next three subsections: $\Phi$ does not vanish; $\Phi_\eta$ does not vanish; 
$\{ \Phi = 0 \}$ and $\{ \Phi_\eta = 0 \}$ are curves intersecting transversally. Additional assumptions will be detailed in the relevant subsections.

\subsubsection{Asymptotics for the linear Cauchy problem}

They are obtained easily by stationary phase (see for instance \cite{Stein}).

\begin{lem} 
\label{lemlin}
Assume that $F \in \mathcal{S}$ is such that $\widehat{F}$ is compactly supported; and suppose that $a''$ does not vanish on $\operatorname{Supp} F$. Then
$$
e^{ita(D)} F (x) = e^{it\left[ a(\xi_0) + X \xi_0 \right]} e^{i\frac{\pi}{4}\sigma} \frac{1}{\sqrt{|a''(\xi_0)|} } \frac{1}{\sqrt{t}} \widehat{F}(\xi_0) + O \left( \frac{1}{t} \right),
$$
where
$$
X \overset{def}{=} \frac{x}{t} \quad , \quad a'(\xi_0)+X \overset{def}{=}0 \quad \mbox{and} \quad \sigma \overset{def}{=} \operatorname{sign} a''(\xi_0).
$$
\end{lem}

\subsubsection{The point of view of stationary phase}
\label{sp}

The solution of~(\ref{eq:dispersive}) is given by
$$
u(t,x) = -\frac{i}{\sqrt{2\pi}} \int_0^t \int \int e^{ix\xi} e^{i\left[ (t-s)a(\xi)+s b (\xi-\eta)+s c(\eta) \right]} \mu(\xi,\eta) \widehat{f}(\xi-\eta) \widehat{g}(\eta) 
\,d\eta \,d\xi\,ds.
$$
Recalling that $X \overset{def}{=} \frac{x}{t}$ and $\mu(\xi,\eta)\overset{def}{=} m(\xi-\eta,\eta)$, the above is equal to
$$
u(t,x) = -\frac{i}{\sqrt{2\pi}} t \int_0^1 \int \int e^{it \left[ (1-\sigma)a(\xi)+\sigma b(\xi-\eta)+\sigma c(\eta) + X \xi\right]} \mu(\xi,\eta) \widehat{f}(\xi-\eta) 
\widehat{g}(\eta) \,d\eta \,d\xi\,d\sigma.
$$
The above is now a (non-standard) stationary phase problem, with a phase given by
$$
\psi(\xi,\eta,\sigma) \overset{def}{=} (1-\sigma)a(\xi)+\sigma b(\xi-\eta)+\sigma c(\eta) + X \xi = a(\xi) + \sigma \Phi(\xi,\eta) +X \xi.
$$
The phase of the gradient is given by
$$
\nabla_{\xi,\eta,\sigma}  \psi = \left( \begin{array}{l}  a' + \sigma \Phi_\xi + X  \\ \sigma \Phi_\eta \\ \Phi \end{array} \right).
$$
It vanishes if
\begin{equation}
\label{pinson}
\mbox{either} \quad \left\{ \begin{array}{l} \sigma = 0 \\ \Phi = 0 \\ a'+X = 0 \end{array} \right. \quad \mbox{or} \quad  \left\{ \begin{array}{l} \Phi = 0 \\ \Phi_\eta = 0 \\ a'+\sigma \Phi_\xi + X = 0. \end{array} \right.
\end{equation}
The Hessian of $\psi$ is given by
$$
\operatorname{Hess}_{\xi,\eta,\sigma} \psi = \left( \begin{array}{lll}  a'' + \sigma \Phi_{\xi \xi} & \sigma \Phi_{\xi \eta} &  \Phi_\xi \\   \sigma \Phi_{\xi \eta} & \sigma \Phi_{\eta \eta} & \Phi_\eta \\  \Phi_\xi & \Phi_\eta & 0 \end{array} \right).
$$
On stationary points of the first type in~(\ref{pinson}), the Hessian is degenerate if and only if $(\xi,\eta)$ belongs to the space-time resonant set.
On stationary points of the second type in~(\ref{pinson}), the Hessian is generically non-degenerate. 

The main difficulty in the analysis below will be to handle the stationary points on the boundary of the integration domain, namely those for which $\sigma=0$ or $1$; 
this will be even more complicated when they are degenerate.

\subsection{In the absence of time resonances}

\begin{thm}
\label{albatros1}
Assume that $\Phi(\xi,\eta)$ does not vanish on $\operatorname{Supp} m$ (ie $\Gamma=\emptyset$), and that $f$ and $g$ belong to $\mathcal{S}$. Then, as $t\rightarrow \infty$,
$$
u(t) = e^{ita(D)} F + O \left( \frac{1}{t} \right).
$$
with
$$
F = T_{\frac{m}{\phi}}(f,g).
$$
\end{thm}

\begin{rem}
The asymptotic behavior of $e^{ita(D)} F$ is given by Lemma~\ref{lemlin}.
\end{rem}

\begin{proof}
The proof is very easy: $u$ is given by
$$
u(t) = -i e^{ita(D)} \int_0^t T_{me^{is\phi}} (f,g)\,ds
$$
or
$$
u(t) = - T_{\frac{m}{\phi}} (e^{itb(D)}f,e^{itc(D)}g) + e^{ita(D)} T_{\frac{m}{\phi}} (f,g). 
$$
The theorem follows since the first term above is $O(\frac{1}{t})$ by the linear decay estimates.
\end{proof}

\subsection{In the absence of space resonances}

\begin{thm}
\label{albatros2}
Assume that $\psi_\eta$ does not vanish on $\operatorname{Supp} m$ (ie $\Delta=\emptyset$), that $\psi_{\xi \xi}(\xi,\eta,\sigma)$ does not vanish on $\operatorname{Supp} m \times [0,1]$,
and that $f$, $g$ belong to $\mathcal{S}$. Fix $M>0$ and $N\in \mathbb{N}$. Then, as $t\rightarrow \infty$,
$$
u(t) = e^{ita(D)} F + O \left( \frac{1}{M^N \sqrt{t}} \right)
$$
where
$$
F = -i \int_0^M e^{i(t-s)a(D)} T_m(e^{is b (D)} f,e^{is c(D)} g) \,ds
$$
(in other words, $e^{ita(D)} F$ is the solution of $\left\{ \begin{array}{l} i\partial_t u + a(D) u = T_{m}\left( v , w \right) \vsp \\ i\partial_t v + b(D) v = 0 \vsp 
\\ i\partial_t w + c(D) w = 0 \vsp \end{array} \right.$ if $0<t<M$, and $i\partial_t u + a(D) u = 0$ if $t>M$, with the data $u(t=0) = 0$, $v(t=0) = f$, and $ w (t=0) = g$).
\end{thm}

\begin{proof}
Starting from the stationary phase formulation given in Section~\ref{sp}, it suffices to show that
\begin{equation}
\label{pingouin}
 \int_{M/t}^1 \int \int e^{it \psi(\xi,\eta,\sigma)} \mu(\xi,\eta) \widehat{f}(\xi-\eta) \widehat{g}(\eta) \,d\eta \,d\xi\,d\sigma
\end{equation}
is $O\left(\frac{1}{M^N t^{\frac{3}{2}}}\right)$.

First apply the stationary phase lemma in $\xi$ in the above. The vanishing set of $\psi_\xi$ depends on $X$. If $X$ is such that $\psi_\xi$ does not vanish, (\ref{pingouin}) is 
$O\left( \frac{1}{t^N}\right)$ for any $N$ and we are done. Otherwise, $\psi_\xi$ vanishes for some $\xi$, which we denote $\xi_0$, and which is a function of $X$, $\eta$, 
and $\sigma$. We can assume without loss of generality that $\xi_0$ is unique. Since $\psi_{\xi \xi}$ does not vanish by assumption, the stationary phase lemma gives
$$
(\ref{pingouin}) = \int_{M/t}^1 \int e^{it\psi(\xi_0,\eta,\sigma)} \left[ \frac{\alpha(\xi,\eta,\sigma)}{\sqrt{t}} + \frac{\beta(\xi,\eta,\sigma)}{t} + 
\frac{\gamma(\xi,\eta,\sigma)}{t\sqrt{t}} + O\left( \frac{1}{t^2} \right) \right]\,d\eta\,d\sigma,
$$
where $\alpha$, $\beta$ and $\gamma$ are smooth functions which we do not specify. The fourth summand in~(\ref{pingouin}) is already small enough. We will now show how to deal 
with the first one, and this will conclude the proof since the second and third ones are easier (better decay). Thus we now want to show that
\begin{equation}
\label{pingouin2}
\int_{M/t}^1 \int e^{it\psi(\xi_0,\eta,\sigma)} \frac{\alpha(\xi,\eta,\sigma)}{\sqrt{t}} \,d\eta\,d\sigma 
\end{equation}
is $O\left( \frac{1}{M^N t} \right)$. In order to take advantage of oscillations in $\eta$ observe that
$$
\partial_\eta \left[\psi(\xi_0(\eta,\sigma,X),\eta,\sigma)\right] = \partial_\eta \xi_0 \left[ \partial_\xi \psi \right] (\xi_0,\eta,\sigma) + \left[ \partial_\eta \psi \right] 
(\xi_0,\eta,\sigma) = \left[ \partial_\eta \psi \right] (\xi_0,\eta,\sigma) = \sigma \left[ \partial_\eta \Phi \right] (\xi_0,\eta).
$$
By hypothesis, $\partial_\eta \Phi$ does not vanish, therefore
$$
\left| \partial_\eta \left[\psi(\xi_0,\eta,\sigma)\right] \right| \gtrsim \sigma.
$$
Integrating by parts $N+1$ times with the help of the identity $\frac{1}{t \partial_\eta \left[\psi(\xi_0,\eta,\sigma)\right]} \partial_{\eta}  e^{it\psi(\xi_0,\eta,\sigma)} 
= i e^{it\psi(\xi_0,\eta,\sigma)}$, we obtain
$$
\left| (\ref{pingouin2}) \right| \lesssim \int_{M/t}^1 \frac{1}{(\sigma t)^{N+1}\sqrt{t}} \,d\sigma \lesssim \frac{1}{M^N t^{\frac{3}{2}}},
$$
which concludes the proof.
\end{proof}

\subsection{Space-time resonance set reduced to a point}

\begin{thm}
\label{goeland}
Assume that $f$, $g$ belong to $\mathcal{S}$, that there exists a unique $(\xi_0,\eta_0)$ such that
$$
\Phi(\xi_0,\eta_0) = \Phi_\eta(\xi_0,\eta_0) = 0,
$$
and that the following technical, generic hypotheses are satisfied:
\begin{itemize}
\item (we are under the standing assumption $(H)$, but only the non-vanishing of $a''$ is used here)             
\item $\Phi_\xi(\xi_0,\eta_0) \neq 0$
\item $\Phi_{\eta \eta}(\xi_0,\eta_0) \neq 0$,
\end{itemize}
and that $\operatorname{Supp} m$ is contained in a small enough neighbourhood of $(\xi_0,\eta_0)$.

\medskip

Recall that $X = \frac{x}{t}$, and set $\Sigma(X) \overset{def}{=} - \frac{1}{\Phi_\xi(\xi_0,\eta_0)} (a'(\xi_0)+X)$. Let $\epsilon>0$ be small enough. 
Assume without loss of generality that $\Phi_\xi(\xi_0,\eta_0) > 0$. Then
\begin{itemize}
\item If $X< - \Phi_\xi(\xi_0,\eta_0) - a'(\xi_0) - \epsilon$,
$$
u(t) = O\left(\frac{1}{t^N} \right)
$$
for any $N$.
\item If $- \Phi_\xi(\xi_0,\eta_0) - a'(\xi_0) - \epsilon <X< - \Phi_\xi(\xi_0,\eta_0) - a'(\xi_0) + \epsilon$,
$$
\displaystyle u(t) = \frac{1}{\sqrt{t}} A_2(\Sigma) \mathcal{G}_1(\sqrt{t}\left[\Sigma-1\right]) + O\left(\frac{1}{t}\right).
$$
for a smooth function $A_2$.
\item If $- \Phi_\xi(\xi_0,\eta_0) - a'(\xi_0) + \epsilon <X< - a'(\xi_0) - \epsilon$,
$$
\displaystyle u(t,x) = \frac{1}{\sqrt{t}} \frac{A_1}{\sqrt{\Sigma(X)}} e^{it\psi(\xi_0,\eta_0,\Sigma)} 
+ O\left(\frac{1}{t} \right)
$$
for a constant $A_1$.
\item If $- a'(\xi_0) - \epsilon < X <  - a'(\xi_0) + \epsilon$,
$$
u(t) = A_0(\Sigma) \frac{1}{t^{1/4}} \mathcal{G}_2 (\sqrt{t}\Sigma) 
+ \left\{ \begin{array}{ll} O(t^{-3/4}) & \mbox{if $|\sqrt{t}\Sigma|<1$} \\ O\left( \frac{|\log t|}{\sqrt{t}} \right) & \mbox{if $|\sqrt{t}\Sigma|>1$} . \end{array} \right.
$$
for a smooth function $A_0$.
\item If $- a'(\xi_0) + \epsilon < X$,
$$
u(t) = O\left(\frac{1}{t^N} \right).
$$ 
for any $N$.
\end{itemize}
\end{thm}

\begin{rem} 
\begin{enumerate}
\item The above theorem provides an efficient equivalent of $u(t)$ for large $t$ in all the zones of the space-time plane $(x,t)$, except where $\Sigma$ is small, but larger that 
$\frac{1}{|\log t|^2}$ (for then $\frac{|\log t|}{\sqrt{t}} > \frac{1}{t^{1/4}} |\mathcal{G}_2 (\sqrt{t}\Sigma)|$). Dealing with this region would require fairly technical 
further developments, from which we refrain.
\item If $\Phi$ vanishes at order 1 on $\Gamma$ and $\Delta$, the conditions $\Phi_\xi(\xi_0,\eta_0) \neq 0$ and $\Phi_{\eta \eta}(\xi_0,\eta_0) \neq 0$ are equivalent to 
$\Gamma$ and $\Delta$ intersecting transversally at $(\xi_0,\eta_0)$. Indeed, a tangent vector to $\Gamma$ (respectively, $\Delta$) at $(\xi_0,\eta_0)$ is given by $\left( 
\begin{array}{c} \partial_\eta \Phi(\xi_0,\eta_0) \\ - \partial_\xi \Phi(\xi_0,\eta_0) \end{array} \right) = \left( \begin{array}{c} 0 \\ - \partial_\xi \Phi(\xi_0,\eta_0) \end{array} \right)$ 
(respectively $\left( \begin{array}{c} \partial_\eta^2 \Phi(\xi_0,\eta_0) \\ - \partial_\eta \partial_\xi \Phi(\xi_0,\eta_0) \end{array} \right)$). These two vectors are
not colinear if $\partial_\xi \Phi(\xi_0,\eta_0) \partial_\eta^2 \Phi(\xi_0,\eta_0) \neq 0$.
\item The hypothesis that $\operatorname{Supp} m$ be restricted in a small enough neighbourhood is not restrictive: away from $(\xi_0,\eta_0)$, either $\Phi$ or $\Phi_\eta$ is
non zero, so either Theorem~\ref{albatros1} or Theorem~\ref{albatros2} applies.
\end{enumerate}
\end{rem}

The proof distinguishes between the three regions: $\sigma$ away from $0$ and $1$, $\sigma$ close to $0$, and $\sigma$ close to $1$. Starting from the formula derived in 
Section~\ref{sp}, we split the time integral as follows:
\begin{equation}
\begin{split}
u(t,x) & =  -\frac{i}{\sqrt{2\pi}} t \int_0^1 \int \int e^{it \psi(\xi,\eta,\sigma)} \mu(\xi,\eta) \widehat{f}(\xi-\eta) \widehat{g}(\eta) \,d\xi \,d\eta \,d\sigma \\
& =  -\frac{i}{\sqrt{2\pi}} t \int_0^t \int \int \left[ \chi_I(\sigma) + \chi_{II} (\sigma) + \chi_{III}(\sigma) \right] \dots \,d\xi\,d\eta\,d\sigma \\
& \overset{def}{=} I + II + III.
\end{split}
\end{equation}
In the above, $\chi_I$, $\chi_{II}$ and $\chi_{III}$ are three smooth, positive functions, adding up to one for each $\sigma$ and such that
\begin{itemize}
\item $\chi_{II} (\sigma) = 0$ if $\sigma<\delta$ and $1$ if $\sigma>2\delta$
\item $\chi_{I} (\sigma) = 0$ if $\sigma<\delta$ or $\sigma>1-\delta$, and $1$ if $2\delta<\sigma<1-2\delta$
\item $\chi_{III} (\sigma) = 0$ if $\sigma<1-2\delta$ and $1$ if $\sigma>1-\delta$.
\end{itemize}
In the above, $\delta > 0$ is a sufficiently small number.

\subsubsection{The contribution of $\sigma$ away from $0$ and $1$}

This is the simplest case since it can be settled by resorting to elementary stationary phase considerations. Our aim is to estimate
$$
I = -\frac{i}{\sqrt{2\pi}} t \int_0^1 \int \int \chi_I(\sigma) e^{it \psi(\xi,\eta,\sigma)} \mu(\xi,\eta) \widehat{f}(\xi-\eta) \widehat{g}(\eta) \,d\xi \,d\eta \,d\sigma.
$$
The phase $\psi(\xi,\eta,\sigma)$ is also a function of $X$, but we consider from now on that $X$ is fixed.

Since $\sigma$ does not vanish on $\operatorname{Supp} \chi_I$, the gradient 
$\nabla_{\xi,\eta,\sigma} \psi = \left( \begin{array}{l}   a' + \sigma \Phi_\xi + X  \\ \sigma \Phi_\eta \\ \Phi  \end{array} \right)$ vanishes if
$$
\Phi(\xi,\eta) = \Phi_\eta(\xi,\eta) = 0 \quad \mbox{and} \quad a'(\xi) + \sigma \Phi_\xi(\xi,\eta) + X = 0.
$$
The first two conditions impose $(\xi,\eta) = (\xi_0,\eta_0)$ whereas the third one gives 
$$\sigma = \Sigma(X) \overset{def}{=} - \frac{1}{\Phi_\xi(\xi_0,\eta_0)} \left(a'(\xi_0) + X \right)$$
(which makes sense under the assumption that $\Phi_\xi(\xi_0,\eta_0) \neq 0$).
We assume that $X$ is such that $\sigma$ given by the above line lies in $\operatorname{Supp}m$; if this is not the case, the contribution of $I$ is negligible.
The Hessian at $(\Sigma,\xi_0,\eta_0)$ is
$$
\operatorname{Hess}_{\xi,\eta,\sigma} \psi(\xi_0,\eta_0,\Sigma) = \left( \begin{array}{lll}   a'' + \Sigma \Phi_{\xi \xi}(\xi_0,\eta_0)  
& \Sigma \Phi_{\xi \eta}(\xi_0,\eta_0) & \Phi_\xi(\xi_0,\eta_0)  \\  \Sigma \Phi_{\xi \eta}(\xi_0,\eta_0) & \Sigma \Phi_{\eta \eta}(\xi_0,\eta_0) & 0 \\  \Phi_\xi(\xi_0,\eta_0) & 0 & 0 \end{array} \right).
$$
with a determinant
$$
\operatorname{det} \operatorname{Hess}_{\xi,\eta,\sigma} \psi(\xi_0,\eta_0,\Sigma) = -\Sigma \Phi_\xi(\xi_0,\eta_0)^2 \Phi_{\eta \eta}(\xi_0,\eta_0).
$$
Let us assume that $\Phi_{\eta \eta}(\xi_0,\eta_0)$ is not zero, which is generically satisfied. The stationary phase principle gives then (\cite{Stein})
$$
u(t,x) = \frac{1}{\sqrt{t}} \frac{\chi_I(\Sigma(X))}{\sqrt{\Sigma(X)}} e^{it\psi(\xi_0,\eta_0,\Sigma)} A_1 + O \left(\frac{1}{t}\right)
$$
with 
$$
A_1 \overset{def}{=} \frac{(2\pi)^{3/2} e^{i\frac{\pi}{4}S}}{|\Phi_\xi(\xi_0,\eta_0)| \sqrt{|\Phi_{\eta \eta}(\xi_0,\eta_0)|}} \mu(\xi_0,\eta_0) \widehat{f}(\xi_0-\eta_0) \widehat{g}(\eta_0)
$$
where $S$ is the signature of $\operatorname{Hess}_{\xi,\eta,\sigma} \psi(\xi_0,\eta_0,\Sigma)$.

\subsubsection{The contribution of $\sigma$ close to $0$}

\noindent \underline{Step 1: splitting between small and large times}
Our aim is to estimate
$$
II = -\frac{i}{\sqrt{2\pi}} t \int_0^1 \int \int \chi_{II}(\sigma) e^{it \psi(\xi,\eta,\sigma)} \mu(\xi,\eta) \widehat{f}(\xi-\eta) \widehat{g}(\eta) \,d\xi \,d\eta\,d\sigma.
$$
which we split into
$$
II = -\frac{i}{\sqrt{2\pi}} t \left[ \int_0^{1/t} + \int_{1/t}^1 \dots \,d\sigma \right] \overset{def}{=} II_1 + II_2.
$$
Rescaling $II_1$, we see that it can be written
$$
u(t) = e^{ita(D)} F \quad \mbox{where} \quad F = -\frac{i}{\sqrt{2\pi}} \int_0^1 \int e^{i(s-t)a(D)} T_m(e^{-is b (D)} f,e^{-is c(D)} g) \,ds,
$$
so that it reduces to a linear solution for $t$ sufficiently large. We now focus on $II_2$.

\noindent \underline{Step 2: stationary phase in $\xi$}

We want to apply the stationary phase lemma in the variable $\xi$. Observe that
$$
\partial_\xi \psi(\xi,\eta,\sigma) = a'(\xi) + \sigma \Phi_\xi + X 
$$
Thus for $\eta$, $\sigma$ and $X$ fixed, $\partial_\xi \psi(\xi,\eta,\sigma)=0$ might, or not, have a solution in $\operatorname{Supp}m$. 
If not, the contribution is negligible, so let us assume that this equation has a solution $\xi= \Xi(X,\eta,\sigma)$. Next,
$$
\partial_\xi^2 \psi(\xi,\eta,\sigma) = a''(\xi) + \sigma \Phi_{\xi \xi}.
$$
Since we are assuming that $a''$ does not vanish, taking $\delta$ small enough, we can thus ensure that $\partial_\xi^2 \psi(\xi,\eta,\sigma)$ does not vanish.
Applying the stationary phase lemma gives then
\begin{equation}
\label{sterne}
II_2 = t \int_{1/t}^1 \int G(\Xi,\eta) e^{it \psi(\Xi,\eta,\sigma)} 
\left[ \frac {\sqrt{2\pi}e^{iS_0\frac{\pi}{4}}} {\sqrt{\psi_{\xi \xi}(\Xi,\eta,\sigma)} \sqrt{t}} + \frac {\alpha(\eta,\sigma)} {t} + 
\frac{\beta(\eta,\sigma)} {t\sqrt{t}} + O\left( \frac{1}{t^2} \right) \right] \,d\eta\,d\sigma,
\end{equation}
where $S_0=\operatorname{sign}(\psi_{\xi \xi}(\Xi,\eta,\sigma))$, $\alpha$ and $\beta$ are smooth functions, and we denoted for simplicity
$$
G(\xi,\eta,\sigma) = -\frac{i}{\sqrt{2\pi}} \chi_{II}(\sigma) \mu(\xi,\eta) \widehat{f}(\xi-\eta) \widehat{g}(\eta).
$$
The last term in~(\ref{sterne}), containing $O\left( \frac{1}{t^2} \right)$, contributes $O\left( \frac{1}{t^2} \right)$ to $u$; thus we can discard it, and focus on
\begin{equation}
\label{sternearctique}
t \int_{1/t}^1 \int G(\Xi,\eta) e^{it \psi(\Xi,\eta,\sigma)} 
\left[ \frac {\sqrt{2\pi}e^{iS_0\frac{\pi}{4}}} {\sqrt{\psi_{\xi \xi}(\Xi,\eta,\sigma)} \sqrt{t}} + \frac {\alpha(\xi,\eta,\sigma)} {t} + 
\frac{\beta(\xi,\eta,\sigma)} {t\sqrt{t}} \right] \,d\eta\,d\sigma.
\end{equation}

\bigskip \noindent \underline{Step 3: stationary phase in $\eta$} Observe that
$$
\partial_\eta \left[ \psi(\Xi(\eta,\sigma),\eta,\sigma) \right] = \partial_\eta \Xi \left[ \partial_\xi \psi \right] (\Xi,\eta,\sigma) + 
\left[ \partial_\eta \psi \right] (\Xi,\eta,\sigma) = \left[ \partial_\eta \psi \right] (\Xi,\eta,\sigma) = \sigma \left[ \partial_\eta \Phi \right] (\Xi,\eta)
$$
Just like for the stationary phase in $\xi$, we denote $\eta = H(\sigma,X)$ the solution of $\left[ \partial_\eta \Phi \right] (\Xi,\eta) = 0$ (if no solution exist, 
the contribution is negligible). Next, set
$$
\partial_\eta^2 \left[ \psi(\Xi(\eta,\sigma),\eta,\sigma) \right] = \sigma \partial_\eta \Xi \left[ \partial_\xi \partial_\eta \Phi \right] (\Xi,\eta) + 
\sigma \left[ \partial_\eta^2 \Phi \right] (\Xi,\eta) \overset{def}{=} \sigma Z(\eta,\sigma).
$$
We need to assume here that 
$$
Z(\eta,\sigma) \neq 0
$$
if $(\sigma,\Xi,\eta) \in \operatorname{Supp} m \chi_{II}$. 
Since the support of $m$, as well as $\delta$, are assumed to be small enough, it suffices that $Z(\eta_0,0) \neq 0$; but a simple computation reveals that 
$Z(\eta_0,0) = \phi_{\eta \eta}(\eta_0,\xi_0)$, which is non zero by hypothesis. The stationary phase lemma in $\eta$ applied to~(\ref{goeland})gives then
\begin{equation}
\label{goeland2}
\begin{split}
(\ref{sternearctique}) & = t \int_{1/t}^1 G(\Xi,H) e^{it \psi(\Xi,H,\sigma)} \frac{\sqrt{2\pi}e^{iS_1\frac{\pi}{4}}}{\sqrt{t} \sqrt{\sigma} Z(H,\sigma)} 
\left[ \frac {\sqrt{2\pi}e^{iS_0\frac{\pi}{4}}} {\sqrt{\psi_{\xi \xi}(\Xi,H,\sigma)} \sqrt{t}} + \frac {\alpha(H,\sigma)} {t} + 
\frac{\beta(H,\sigma)} {t\sqrt{t}} \right.\\
&\quad \quad \quad \quad\quad \quad\quad \quad\quad \quad\quad \quad \left. + O\left( \frac{1}{t \sqrt{t} \sigma} \right) \right] \,d\sigma,
\end{split}
\end{equation}
where $S_1 = \operatorname{sign}(Z(H,\sigma))$, $\widetilde{\alpha}$ and $\widetilde{\beta}$ are smooth functions. The last summand in~(\ref{goeland2}) contributes
$$
O \left( \frac{t}{t\sqrt{t}} \int_{1/t}^1 \frac{d\sigma}{\sigma} \right) = O \left(\frac{\log t}{\sqrt{t}}\right),
$$
we discard it and focus on
\begin{equation}
\label{goeland3}
t \int_{1/t}^1 G(\Xi,H) e^{it \psi(\Xi,H,\sigma)} \frac{\sqrt{2\pi}e^{iS_1\frac{\pi}{4}}}{\sqrt{t} \sqrt{\sigma} Z(H,\sigma)} \left[ \frac {\sqrt{2\pi}e^{iS_0\frac{\pi}{4}}} 
{\sqrt{\psi_{\xi \xi}(H,\sigma)} \sqrt{t}} + \frac {\alpha(H,\sigma)} {t} + 
\frac{\beta(\xi,H,\sigma)} {t\sqrt{t}} \right] \,d\sigma,
\end{equation}

\bigskip \noindent \underline{Step 4: stationary phase in $\sigma$} In this final step, we are not going to apply the standard stationary phase lemma, but rather its variant given
in Proposition 6.2. Differentiating in $\sigma$ the phase in~(\ref{goeland2}) gives
$$
\partial_\sigma \left[ \psi(\Xi(H(\eta,\sigma),\sigma),H(\sigma),\sigma) \right] = \left[ \partial_\sigma \psi \right] (\Xi,H,\sigma) = \Phi(\Xi(H,\sigma),H(\sigma))
$$
since $\partial_\xi \psi = \partial_\eta \psi = 0$ at the point $(\Xi,H,\sigma)$. Thus $\partial_\sigma \phi = 0$ if $\Phi(\Xi,H) = 0$. Since on the other hand $\partial_\eta 
\Phi(\Xi,H) = 0$ by definition of $H$,
$$
\Phi(\Xi(H,\sigma),H(\sigma)) = 0 \quad \mbox{if and only if} \quad H(\sigma) = \eta_0 \quad \mbox{and} \quad \Xi(\eta_0,\sigma) = \xi_0.
$$
But by definition of $\Xi$ this implies
$$
\sigma = \Sigma(X) = -\frac{X + a'(\xi_0)}{\psi_\xi(\xi_0,\eta_0)}.
$$
In order to apply Proposition~\ref{fregate}, we need to check that
$$
\partial_\sigma^2 \left[ \psi(\Xi(H(\eta,\sigma),\sigma),H(\sigma),\sigma) \right] (\Sigma) = \partial_\sigma \left[ \Phi(\Xi(H,\sigma),H(\sigma)) \right] (\Sigma) \neq 0.
$$
Since $\delta$ is chosen small enough, it suffices to check that it holds for $\Sigma = 0$ (that is, when $X$ is such that $\Sigma(X)=0$). This follows from the following 
computation:
\begin{equation*}
\begin{split}
\partial_\sigma \left[ \Phi(\Xi(H,\Sigma),H(\Sigma)) \right] (0) & = \partial_\xi \Phi(\xi_0,\eta_0) \left[ \partial_\sigma \Xi(\eta_0,0) + \partial_\sigma H(0) 
\partial_\eta \Xi (\eta_0,0) \right] + \partial_\eta \Phi(\xi_0,\eta_0) \partial_\sigma H(0) \\
& = \partial_\xi \Phi(\xi_0,\eta_0) \partial_\sigma \Xi(\eta_0,0) = - \frac{\phi_\xi(\xi_0,\eta_0)^2}{a''(\xi_0)} \neq 0,
\end{split}
\end{equation*}
where we used that $\partial_\eta \Phi(\xi_0,\eta_0) = \partial_\eta \Xi (\eta_0,0) = 0$ and $\partial_\sigma \Xi(\eta_0,0) = - \frac{\phi_\xi(\xi_0,\eta_0)}{a''(\xi_0)}$. 
We now write 
$$
(\ref{goeland3}) = t \int_{1/t}^{1} \dots \,d\sigma = t \int_0^1 - t \int_0^{1/t} \dots \,d\sigma.
$$
The second summand, $t \int_0^{1/t} \dots \,d\sigma$, is directly estimated to be $O\left(\frac{1}{\sqrt{t}}\right)$. As for the first summand, $t \int_0^1\dots \,d\sigma$, apply
Proposition~\ref{fregate} $(iv)$ to obtain
$$
(\ref{goeland3}) = A_0(\Sigma) \mathcal{G}_2(\sqrt{t}\Sigma) +
\left\{ \begin{array}{ll} O(t^{-3/4}) & \mbox{if $|\sqrt{t}\Sigma|<1$} \\ O\left( \sqrt{\frac{|\Sigma|}{t}} \right) & \mbox{if $|\sqrt{t}\Sigma|>1$}, \end{array} \right.
$$
where $A_0$ is a smooth function which we do not detail here.

\subsubsection{The contribution of $\sigma$ close to $1$}

In order to estimate
$$
III = -\frac{i}{\sqrt{2\pi}} t \int_0^1 \int \int \chi_{III}(\sigma) e^{it \psi(\xi,\eta,\sigma)} \mu(\xi,\eta) \widehat{f}(\xi-\eta) \widehat{g}(\eta) \,d\eta \,d\xi\,d\sigma.
$$
an approach similar to the one used for $II$ can be followed, the details being more simple: first apply the stationary phase Lemma in the $(\xi,\eta)$ variables, and then
Proposition~\ref{fregate} $(i)$. We do not give the details here.

\subsection{Conclusion}

\subsubsection{Space-time localization of the waves}
As a conclusion of the asymptotic analysis of waves which has just been carried out, it is interesting to compare the space-time localisations of the emerging wave $u$, 
solution of~(\ref{eq:dispersive}), in the three situations which we examined. To simplify, suppose that $f$ and $g$ are localized in space close to 0, and in frequency close to 
$\widetilde{\xi} - \widetilde{\eta}$ and $\widetilde{\eta}$ respectively. Then
\begin{itemize}
\item In the absence of space-time resonances, $u$ will be localized where $X \sim -a'(\widetilde{\xi})$, where it will have size $\sim \frac{1}{\sqrt{t}}$.
\item If the space-time resonant set is reduced to a point, and under the assumptions of Theorem~\ref{goeland}, $u$ will have size $\sim \frac{1}{t^{1/4}}$ if 
$-\Phi_\xi(\eta_0,\xi_0) -a'(\xi_0) <X<- a'(\xi_0)$, and size $\sim \frac{1}{\sqrt{t}}$ if $X \sim -a'(\xi_0)$.
\end{itemize}

\subsubsection{Lower bound} \label{subsub:lower} The asymptotic equivalents which have been computed also provide lower bounds for $L^p$ norms of $u$. In the absence of space time resonances, we do not
learn anything, since the equivalent for $u$ is similar to a linear solution. However, in the case when Theorem~\ref{goeland} applies (ie when $\Delta$ and 
$\Gamma$ intersect transversally at a point), we get for $t$ large
$$
\|u(t)\|_{L^q} \gtrsim 
\left\{
\begin{array}{l} 
\log t \quad \mbox{for $q=2$} \\
t^{\frac{1}{2q}-\frac{1}{4}} \quad \mbox{for $2 < q \leq \infty$}.
\end{array}
\right.
$$

\section{Non-localized data}

\label{nld}

In this section, the data are only supposed to belong to $L^2$, as opposed to the next one, where the data will belong to weighted $L^2$ spaces.

\subsection{Main results}

\begin{thm} 
\label{mainth} Simply assume that $m$ is smooth and compactly supported and $a,b,c$ are real-valued. In the different following situations, for $q\in [2,\infty]$, the solution $u$ of~(\ref{eq:dispersive}) satisfies
$$ \left\| u(t) \right\|_{L^q} \lesssim \alpha(t) \|f\|_{L^2} \|g\|_{L^2} $$
with $\alpha(t)$ as follows: \\
\begin{center}
\renewcommand{\arraystretch}{1.8}
\begin{tabular}{|l|c|}
  \hline
     & $\alpha(t)$   \\
 \hline
In general & $t$  \\
 \hline
 $\Gamma$ is empty & $1$  \\
\hline
 $\Gamma$ is a point where $\phi$ vanishes at order two & $t^{\frac{1}{2}+\frac{1}{2q}}$  \\ 
\hline
$\Gamma$ is a non-characteristic curve where $\phi$ vanishes at order one  & 
$\begin{array}{ll} t^{\frac{1}{q}} & \mbox{if $2 \leq q <\infty$} \\ \langle \log t \rangle & \mbox{if $q=\infty$} \end{array}$
\\
\hline
$\Gamma$ is a curve with non-vanishing curvature where $\phi$ vanishes at order one & $t^{\frac{1}{4}+\frac{1}{2q}}$  \\
\hline
$\Gamma$ is a general curve where $\phi$ vanishes at order one & $t^{\frac{1}{2}}$.   \\
\hline
\end{tabular}
\end{center}
\end{thm}

In the two first situations above, the bound can be improved if the unitary groups $e^{ita(D)}$, $e^{itb(D)}$ and $e^{itc(D)}$ give decay.

In this setting, using precisely the structure of product of two linear solutions (which cannot be described by only using the set $\phi^{-1}(\{0\})$ as previously), we get the following improvment.

\begin{thm} \label{main2} Assume that $m$ is smooth and compactly supported, and that $(H)$ holds. Then the solution $u$ of~(\ref{eq:dispersive}) satisfies for 
all $q\in[2,\infty]$ 
$$ \left\| u(t) \right\|_{L^q} \lesssim t^{1/2+\frac{1}{2q}} \|f\|_{L^2} \|g\|_{L^2}.$$
If moreover, we assume that $\Gamma=\emptyset$, then for $p,q\in[2,\infty)$ with 
$$\frac{1}{p}+\frac{1}{q} > \frac{1}{2},$$
we get
\begin{equation}
\label{mouette}
\left\|u(t)\right\|_{L^p_t L^q} \lesssim \|f\|_{L^2} \|g\|_{L^2}.
\end{equation}
\end{thm}

\begin{rem}
The last statement of the previous theorem gives decay in an integrated form (belonging of $u$ to some $L^p L^q$), as opposed to the pointwise in time rate of decay obtained earlier; of course, this has to do with the use of Strichartz estimates. Heuristically, (\ref{mouette}) can be understood as giving the rate of decay $\|u(t)\|_{L^q} \lesssim t^{\frac{1}{q}-\frac{1}{2}} \|f\|_{L^2} \|g\|_{L^2}$.
\end{rem}

\bigskip
Then, if the smooth symbol $m$ has not a bounded support, we have the following result

\begin{cor} \label{cor} We want to track the dependence of the bounds in the above theorem on the size of the support of $m$. So
assume that $m$ is bounded by $1$ along with sufficiently many of its derivatives, and that it is supported on $B(0,R)$. Then all the previous 
boundedness results hold with an extra factor $R$.
\end{cor}

\begin{proof} In Theorem \ref{mainth} or \ref{main2}, we have obtained boundedness from $L^2\times L^2$ to ${\mathcal B}$ (where ${\mathcal B}=L^q$ or $L^p_T L^q$) of the operator $T_t= T_\sigma$
with the symbol
$$ \sigma(\xi,\eta) \overset{def}{=} e^{ita(\xi+\eta)} \frac{e^{it\phi(\xi,\eta)}-1}{\phi(\xi,\eta)} m(\xi,\eta), $$ 
when $m$ has a bounded support. \\
So now, considering a smooth symbol $m$ supported on $B(0,R)$, we split it (using a smooth partition of the unity) as
$$ m = \sum_{k,l} m_{k,l}$$
with $m_{k,l}$ smooth symbols supported on $[k-1,k+1] \times [l-1,l+1]$.
Applying the previous results (invariant by modulation), we get 
$$ \| T_\sigma\|_{\mathcal B} \leq \sum_{k,l} \|T_{\sigma_{k,l}}\|_{\mathcal B} \leq c_{\mathcal B} \sum_{k,l} \|\pi_{k} f \|_{L^2} \|\pi_{l}g \|_{L^2},$$
where $c_{\mathcal B}$ is the constant previously obtained for compactly supported symbols and $\pi_{k} f$ is a smooth truncation of $f$ for frequencies around $[k-1,k+1]$.
Using orthogonality, it comes that
$$ \| T_\sigma\|_{\mathcal B} \leq c_{\mathcal B} \left(\sum_{k,l} 1 \right)^{1/2}  \| f \|_{L^2} \|g \|_{L^2},$$
which gives the desired results since $k,l\in\{-R-1,...,R+1\}$.
 \end{proof}

\subsection{Proof of Theorem~\ref{mainth}: the general case}

Using no properties on $\Gamma$ or $a$, $b$, $c$, we can get the following general bound:

\begin{lem} \label{lem:direct} Assume that $m$ is compactly supported, then for all $q\in[2,\infty]$, the solution $u$  to~(\ref{eq:dispersive}) satisfies
$$ \left\| u(t) \right\|_{L^q} \lesssim t\|f\|_{L^2} \|g\|_{L^2}.$$
\end{lem}

\begin{proof} The solution $u(t)$ is given by
\begin{align*}
 u(t)=T_t(f,g)(x) & = \int_{\R^2} e^{ix(\xi+\eta)} e^{ita(\xi+\eta)} \frac{e^{it\phi(\xi,\eta)}-1}{\phi(\xi,\eta)} m(\xi,\eta) \widehat{f}(\xi) \widehat{g}(\eta)\, d\xi d\eta \\
 & = T_\sigma(f,g)(x),
\end{align*}
with the symbol
$$ \sigma(\xi,\eta) \overset{def}{=} e^{ita(\xi+\eta)} \frac{e^{it\phi(\xi,\eta)}-1}{\phi(\xi,\eta)} m(\xi,\eta). $$
In this general setting, we only know that $\sigma$ is bounded by $t$ and compactly supported. Lemma \ref{cormoran} implies 
$$ \|T_t(f,g)\|_{L^q} \lesssim t \|f\|_{L^2} \|g\|_{L^2}.$$
\end{proof}

The aim of the following subsections is to improve on this bound under two different kinds of assumptions
\begin{itemize}
\item On the one hand, using geometric properties of the resonance set $\Gamma$
\item and on the other hand, by assuming linear Strichartz inequalities for the unitary groups $e^{ita(D)}$, $e^{itb(D)}$ and $e^{itc(D)}$ and using the structure of product of two linear solutions.
\end{itemize}

\subsection{Proof of Theorems~\ref{mainth} and \ref{main2}: the case without resonances}

In other words: we assume here that the phase function $\phi$ does not vanish.

\begin{prop} \label{prop:cas1} Assume that $\Gamma=\emptyset$ and that $m$ is compactly supported, then for $q\in[2,\infty]$ the solution $u$ of~(\ref{eq:dispersive}) satisfies
\be{eq:1} \left\| u \right\|_{L^q} \lesssim \|f\|_{L^2} \|g\|_{L^2}. \ee
\end{prop}

\begin{proof}
The solution $u(t)$ is given by
\begin{align*}
 T_t(f,g)(x) & = \int_{\R^2} e^{ix(\xi+\eta)} e^{ita(\xi+\eta)} \frac{e^{it\phi(\xi,\eta)}-1}{\phi(\xi,\eta)} m(\xi,\eta) \widehat{f}(\xi) \widehat{g}(\eta) d\xi d\eta \\
 & = T_\sigma(f,g)(x),
\end{align*}
with the symbol
$$ \sigma(\xi,\eta) \overset{def}{=} e^{ita(\xi+\eta)} \frac{e^{it\phi(\xi,\eta)}-1}{\phi(\xi,\eta)} m(\xi,\eta). $$
Since $a$ is real-valued and $\phi$ is non-vanishing, then $\sigma$ is bounded by a constant and compactly supported. Lemma~\ref{cormoran} yields 
$$ \|T_t(f,g)\|_{L^q} \lesssim \|f\|_{L^2} \|g\|_{L^2}.$$
\end{proof}

Let us now deal with the improved bounds of Theorem \ref{main2} (using dispersive and Strichartz estimates on the linear evolution groups):

\begin{proof}[Proof of Theorem \ref{main2}]
Let us check the first claim. For every $s\in(0,t)$, we use the dispersive inequality (\ref{eq:dispersive}) and the $L^2\times L^2 \rightarrow L^1$ boundedness of $T_m$ to get
\begin{align*}
\left\| e^{ita(D)} T_{m e^{is\phi}} (f,g) \right\|_{L^\infty} & = \left\|e^{i(t-s)a(D)} T_{m} \left(e^{isb(D)}f, e^{isc(D)}g\right) \right\|_{L ^\infty} \\
& \lesssim \frac{1}{\sqrt{t-s}} \|f\|_{L^2} \|g\|_{L^2}.
\end{align*}
Then, integrating for $s\in(0,t)$ it follows that
$$ \left\| T_t (f,g)\right\|_{L^\infty} \lesssim t^{1/2} \|f\|_{L^2} \|g\|_{L^2}.$$
Similarly using the $L^2\times L^\infty \rightarrow L^2$ boundedness of $T_m$, we have for all $s>0$
\begin{align*}
\left\| e^{ita(D)} T_{m e^{is\phi}} (f,g) \right\|_{L^2} & = \left\| T_{m} \left(e^{isb(D)}f, e^{isc(D)}g\right) \right\|_{L ^2} \\
& \lesssim  \|f\|_{L^2} \|e^{isc(D)}g\|_{L^\infty},
\end{align*}
which yields (using Strichartz inequality)
$$ \left\| T_t (f,g) \right\|_{L^2} \lesssim t^{3/4} \|f\|_{L^2} \|g\|_{L^2}.$$
The proof is concluded by interpolating between $L^2$ and $L^\infty$.

Next, assume that $\Gamma=\emptyset$, which means that $\phi$ is not vanishing on the support of $m$. Computing the integration over $s\in[0,t]$, we can split
\begin{align*}
i T_t(f,g)(x) = I_t(f,g)-II_t(f,g),
\end{align*}
with
\begin{align*}
 I_t(f,g)(x) \overset{def}{=}\int_{\R^2} e^{ix(\xi+\eta)}  e^{it(b(\xi)+c(\eta))} \frac{1}{\phi(\xi,\eta)} m(\xi,\eta) \widehat{f}(\xi) \widehat{g}(\eta) d\xi d\eta
\end{align*}
and 
 \begin{align*}
 II_t(f,g)(x) \overset{def}{=}\int_{\R^2} e^{ix(\xi+\eta)}  e^{ita(\xi+\eta)} \frac{1}{\phi(\xi,\eta)} m(\xi,\eta) \widehat{f}(\xi) \widehat{g}(\eta) d\xi d\eta.
\end{align*}
In other words
$$ I_t(f,g) = T_{m/\phi}(e^{itb(D)} f, e^{itc(D)}g)$$
and
$$ II_t = e^{ita(D)} T_{m/\phi}(f,g).$$
Since $\phi$ is assumed to be smooth and non-vanishing, $m/\phi$ is also smooth and compactly supported so that the bilinear operator $T_{m/\phi}$ is bounded from $L^P \times L^Q$ into $L^R$ 
as soon as $\frac{1}{P} + \frac{1}{Q} \geq \frac{1}{R}$. 

Choose now $p$ and $q$ as in the statement of the theorem. Using the dispersive estimates and Bernstein's inequality (indeed, since $m$ has a compact support,
it is possible to assume that $\widehat{f}$ and $\widehat{g}$ are compactly supported) gives
$$
\left\|I_t(f,g)\right\|_{L^p L^q} \lesssim \left\|T_{m/\phi}(f,g)\right\|_{L^1} \lesssim \|f\|_{L^2} \|g\|_{L^2}.
$$
Therefore, $e^{itb(D)} f$ enjoys the usual Strichartz estimates, as well as, by Bernstein's inequality, the bounds 
$\left\| e^{itb(D)} f \right\|_{L^{2q}}  \lesssim \|f\|_{L^q}$; similarly for $g$. This gives
$$
\left\| II_t(f,g) \right\|_{L^p L^q} \lesssim \left\| e^{itb(D)} f \right\|_{L^{2p} L^{q}} \left\| e^{itc(D)} g \right\|_{L^{2p} L^{q}} \lesssim \|f\|_{L^2} \|g\|_{L^2}.
$$
\end{proof}

\subsection{The case with resonance at only one point}

\begin{prop} \label{prop:cas2} Assume that $\phi$ only vanishes at the point $(\xi_0,\eta_0)$. Assume further that
$\nabla \phi$ also vanishes $(\xi_0,\eta_0)$, but that $\operatorname{Hess} \phi$ has a definite sign at that point. 
Then if $q \in [2,\infty]$, the solution $u$ of~(\ref{eq:dispersive}) satisfies
$$
\left\| u(t) \right\|_{L^q} \lesssim t^{\frac{1}{2}+\frac{1}{2q}} \|f\|_{L^2} \|g\|_{L^2}.
$$
\end{prop}

\begin{proof}
Assume for simplicity that $\phi$ vanishes at order 2 at 0. Next take $\chi$ a smooth, compactly supported function, equal to $1$ on $B(0,1)$, and set $\psi = \chi(\cdot/2) - \chi$ so that
$$
1 = \chi + \sum_{j \geq 1} \psi(2^{-j} \cdot).
$$
Decompose then the symbol
\begin{align*}
e^{ita(\xi)} m \frac{e^{it\phi}-1}{\phi} & = \left[ \chi \left( \sqrt{t} (\xi,\eta) \right) + \sum_{j \geq 1} \psi(2^{-j} \sqrt{t} (\xi,\eta)) \right] e^{ita(\xi)} m \frac{e^{it\phi}-1}{\phi} \\
& \overset{def}{=} m_0(\xi,\eta) + \sum_{j \geq 1} m_j(\xi,\eta).
\end{align*}
Obviously,
$$
T_t = T_{m_0} + \sum_{j\geq 1} T_{m_j},
$$
so it suffices to bound the summands above. The symbol $m_0$ (respectively, $m_j$) is supported on a ball of radius $\sim \frac{1}{\sqrt{t}}$ (respectively, $\sim \frac{2^j}{\sqrt{t}}$), and bounded by $t$ (respectively, $t 2^{-2j}$). It follows by Lemma~\ref{cormoran} that
$$
\left\|T_{m_0}\right\|_{L^2 \times L^2 \rightarrow L^q} \lesssim t^{\frac{1}{2} + \frac{1}{2q}} \;\;\;\;\;\mbox{and}\;\;\;\;\; \left\|T_{m_j}\right\|_{L^2 \times L^2 \rightarrow L^q} \lesssim t^{\frac{1}{2} + \frac{1}{2q}} 2^{j(-1-\frac{1}{q})}.
$$
Therefore,
$$
\left\|T_t\right\|_{L^2 \times L^2 \rightarrow L^q} \lesssim t^{\frac{1}{2} + \frac{1}{2q}} \left[ 1 + \sum_{j \geq 1}  2^{j(-1-\frac{1}{q})}\right] \lesssim t^{\frac{1}{2} + \frac{1}{2q}},
$$
which is the desired result.
\end{proof}

\subsection{The case of resonances along a curve}

\begin{prop} \label{prop:cas3} Assume that $\Gamma$ is a smooth curve, where $\phi$ vanishes at order 1.
Then if $q \in [2,\infty]$, the solution $u$ of~(\ref{eq:dispersive}) satisfies
\begin{itemize}
\item if $\Gamma$ is nowhere caracteristic
$$
\left\| u(t) \right\|_{L^q} \lesssim \|f\|_{L^2} \|g\|_{L^2}
\left\{ \begin{array}{ll} t^{\frac{1}{q}} & \mbox{if $2\leq q < \infty$} \\ \langle \log t \rangle & \mbox{if $q = \infty$} \end{array} \right. \,
$$
\item if $\Gamma$ has non-vanishing curvature
$$
\left\| u(t) \right\|_{L^q} \lesssim t^{\frac{1}{4}+\frac{1}{2q}} \|f\|_{L^2} \|g\|_{L^2},
$$
\item else
$$
\left\| u(t) \right\|_{L^q} \lesssim t^{\frac{1}{2}} \|f\|_{L^2} \|g\|_{L^2}.
$$
\end{itemize}
\end{prop}

As explained in Remark \ref{rem:car}, the estimate for a nowhere characteristic curve $\Gamma$ still holds if the only characteristic points are characteristic along the variable $\xi+\eta$.

\begin{proof} We only treat the case where $\Gamma$ is non-characteristic, and $2\leq q <\infty$; the other cases can be obtained by a similar argument.
Similarly to Proposition \ref{prop:cas2}, consider  $\chi$ a smooth, compactly supported function, equal to $1$ on $[0,1]$, and set $\psi = \chi(\cdot/2) - \chi$ 
so that
$$
1 = \chi + \sum_{j \geq 1} \psi(2^{-j} \cdot).
$$
Let denote the distance function $d_\Gamma(\xi,\eta)=d((\xi,\eta),\Gamma)$, since $\Gamma$ is supposed to be a smooth curve and $\nabla \phi$ non-vanishing near $\Gamma$, it follows that
$$ d_\Gamma(\xi,\eta) \simeq |\phi(\xi,\eta)|.$$
Decompose then the symbol
\begin{align*}
e^{ita(\xi)} m(\xi,\eta) \frac{e^{it\phi(\xi,\eta)}-1}{\phi(\xi,\eta)} & = \left[ \chi \left(t \phi(\xi,\eta) \right) + \sum_{j \geq 1} \psi(2^{-j} t \phi(\xi,\eta)) \right] e^{ita(\xi)} m(\xi,\eta) \frac{e^{it\phi(\xi,\eta)}-1}{\phi(\xi,\eta)} \\
& \overset{def}{=} m_0(\xi,\eta) + \sum_{j \geq 1} m_j(\xi,\eta).
\end{align*}
Obviously,
$$
T_t = T_{m_0} + \sum_{j\geq 1} T_{m_j},
$$
so it suffices to bound the summands above. The symbol $m_0$ (respectively, $m_j$) is supported, up to a numerical constant, on a neighborhood $\Gamma_{1/t}$ 
(respectively, $\Gamma_{2^j/t}$), and bounded by $t$ (respectively, $t 2^{-j}$). 
If $\Gamma$ is nowhere caracteristic, it follows by Lemma~\ref{supercormoran} that
$$
\left\|T_{m_0}\right\|_{L^2 \times L^2 \rightarrow L^q} \lesssim t^{\frac{1}{q}} \;\;\;\;\;\mbox{and}\;\;\;\;\; 
\left\|T_{m_j}\right\|_{L^2 \times L^2 \rightarrow L^q} \lesssim t^{\frac{1}{q}} 2^{-\frac{j}{q}}.
$$
Therefore,
$$
\left\|T_t\right\|_{L^2 \times L^2 \rightarrow L^q} \lesssim t^{\frac{1}{q}} \left[ 1 + \sum_{j \geq 1}  2^{-\frac{j}{q}}\right] \lesssim t^{\frac{1}{q}},
$$
which is the desired result.
\end{proof}

\section{Localized data}

\label{ld}

We will now assume that the data belongs to a weighted Sobolev space, and study the decay of the solution of~(\ref{eq:dispersive}).

\subsection{Decay of solutions: the role of time resonances}

\begin{prop} \label{prop:cas2bis} Assume as usual that $m$ is smooth and compactly supported.
Assume next that $\phi$ only vanishes at $(\xi_0,\eta_0)$; that $\nabla \phi$ also vanishes at that point; and that $\operatorname{Hess} \phi$ at that point has a definite sign. 
Then if $q \in [2,\infty]$, the solution $u$ of~(\ref{eq:dispersive}) satisfies 
\begin{itemize}
\item If $0 \leq s < \frac{1}{2}$, $\left\|u(t)\right\|_{L^q} \lesssim t^{\frac{1}{2}+\frac{1}{2q}-s} \|f\|_{L^{2,s}} \|g\|_{L^{2,s}}$.
\item If $s > \frac{1}{2}$ and $q<\infty$, $\left\|u(t)\right\|_{L^q} \lesssim t^{–\frac{1}{2q}} \|f\|_{L^{2,s}} \|g\|_{L^{2,s}}$.
\item If $s > \frac{1}{2}$ and $q=\infty$, $\left\|u(t)\right\|_{L^\infty} \lesssim \langle \log t \rangle \|f\|_{L^{2,s}} \|g\|_{L^{2,s}}$.
\end{itemize}
\end{prop}

\begin{proof}
As previously done in the proof of Proposition \ref{prop:cas2}, we decompose the symbol, thus giving rise to a decomposition
$$ u(t)=T_{m_0}(f,g) + \sum_{j\geq 1} T_{m_j}(f,g).$$ 
The symbol $m_0$ (resp. $m_j$) is supported on a ball of radius $t^{-\frac{1}{2}}$ (resp. $2^j t^{-\frac{1}{2}}$) and is bounded by $t$ (resp. $t2^{-2j}$). Consequently we conclude with Lemma \ref{cormoran}.
\end{proof}

\begin{thm} \label{thm:timereso} Assume that $\phi$ vanishes at first order along a non-characteristic curve $\Gamma$. 
Then for $2 \leq q < \infty$ and $s \geq 0$, the solution $u$ of~(\ref{eq:dispersive}) satisfies the following estimates:
\begin{itemize}
\item If $0 \leq s < \frac{1}{4}$, $\left\|u(t)\right\|_{L^q} \lesssim t^{\frac{1-4s}{q}} \|f\|_{L^{2,s}} \|g\|_{L^{2,s}}$.
\item If $s > \frac{1}{4}$, $\left\|u(t)\right\|_{L^q} \lesssim \langle \log t \rangle \|f\|_{L^{2,s}} \|g\|_{L^{2,s}}$.
\end{itemize}
If $q = \infty$, the solution $u$ of~(\ref{eq:dispersive}) satisfies
$$
\left\| u(t) \right\|_{L^\infty} \lesssim \langle \log t \rangle \|f\|_{L^2} \|g\|_{L^2}.
$$
\end{thm}

\begin{rem} It is interesting to notice that the $L^\infty$ estimate of Proposition~\ref{prop:cas3} does not improve if the data belong to weighted $L^2$ spaces.
Also, notice that the $L^2$ estimate for $s>1/4$ is already as good as allowed by the lower bound in Section~\ref{subsub:lower}: any further assumption on space resonances
will not improve the estimate.
\end{rem}

\begin{proof} Just like in the proof of Proposition~\ref{prop:cas3}, split the symbol
$$
e^{ita(\xi)} m(\xi,\eta) \frac{e^{it\phi(\xi,\eta)}-1}{\phi(\xi,\eta)} = m_0(\xi,\eta) + \sum_{j \geq 1} m_j(\xi,\eta).
$$
Obviously,
$$
T_t = T_{m_0} + \sum_{j\geq 1} T_{m_j},
$$
so it suffices to bound the summands above. The symbol $m_0$ (respectively, $m_j$) is supported, up to a constant, on a neibourghhood $\Gamma_{1/t}$ (respectively, $\Gamma_{2^j/t}$), and bounded by $t$ (respectively, $t 2^{-2j}$). 
Since $\Gamma$ is nowhere caracteristic, it follows by Lemma~\ref{macareux} that for $s < 1/4$
\begin{equation*}
\begin{split}
& \left\|T_{m_j}\right\|_{L^{2,s} \times L^{2,s} \rightarrow L^q} \lesssim \left( t^{-1} 2^{j} \right)^{1-\frac{1}{q} + \frac{4s}{q}} (t2^{-2j}) \\
& \left\|T_{m_0}\right\|_{L^{2,s} \times L^{2,s} \rightarrow L^q} \lesssim t^{-1+\frac{1}{q} - \frac{4s}{q}} t,
\end{split}
\end{equation*}
with corresponding estimates if $s>1/4$. The proof of the proposition is concluded by summing the above bounds for the elementary operators $T_{m_j}$.
\end{proof}

Following the same reasoning and estimates in \cite{BG2}, it is possible to get similar results for a curve admitting characteristic points with non-vanishing curvature. We do not detail this point here.



\subsection{Decay of solutions: the role of space resonances}

\begin{thm} \label{thm:main3} Assume that $(H)$ holds, and that $\Delta = \emptyset$, or in other words that $\left( \partial_\xi - \partial_\eta \right) \phi$ never vanishes.
Then the solution $u$ of~(\ref{eq:dispersive}) satisfies the following bounds: for any $\delta>0$
\begin{itemize}
\item If $0 \leq s < \frac{1}{q}$, $\|u(t)\|_{L^q} \lesssim t^{\frac{1}{2q} + \frac{1}{2} - \frac{3}{2}s + \delta} \|f\|_{L^{2,s}} \|g\|_{L^{2,s}}$.
\item If $\frac{1}{q} < s < 1-\frac{1}{q}$, $\|u(t)\|_{L^q} \lesssim t^{\frac{1}{2}-s + \delta} \|f\|_{L^{2,s}} \|g\|_{L^{2,s}}$.
\item If $s > 1 - \frac{1}{q}$, $\|u(t)\|_{L^q} \lesssim t^{\frac{1}{q}-\frac{1}{2} + \delta} \|f\|_{L^{2,s}} \|g\|_{L^{2,s}}$.
\end{itemize}
\end{thm}

\begin{proof} The proof proceeds by interpolating between the following $L^2$ and the $L^\infty$ estimates. Indeed if $s<\frac{1}{q}$ then for $\theta=\frac{2}{q}$, we 
have $L^q:=(L^2,L^\infty)_\theta$ and $L^{2,s}=(L^{2,\frac{qs}{2}},L^{2,0})_\theta$ with $\frac{qs}{2}\leq \frac{1}{2}$. We conclude similarly for the two other cases.

Recall that $\Phi(\xi,\eta) \overset{def}{=} \phi(\xi-\eta,\eta)$, so that the hypothesis on $\phi$ translates into $\partial_\eta \Phi \neq 0$, and $u$ reads in Fourier space
$$
\widehat{u}(t,\xi) = e^{ita(\xi)} \int_0^t \int e^{i\tau \Phi(\xi,\eta)} \widehat{f}(\xi-\eta) \widehat{g}(\eta) m(\xi-\eta,\eta) \,d\eta\,d\tau.
$$

\bigskip \noindent \underline{The $L^2$ estimate:} We want to prove for every exponent $\delta>0$ (as small as we want)
\begin{equation}
\|u(t)\|_{L^2} \lesssim 
\left\{ \begin{array}{l} 
t^{\frac{3}{4} - \frac{3}{2}s + \delta} \|f\|_{L^{2,s}} \|g\|_{L^{2,s}} \quad \mbox{if $0 \leq s \leq \frac{1}{2}$} \\ 
\|f\|_{L^{2,s}} \|g\|_{L^{2,s}} \quad \mbox{if $s > \frac{1}{2}$}. \end{array} \right. \label{eq:am}
\end{equation}
The result for $s=0$ is given by Theorem \ref{main2}. So let us study the case $s=1/2$ so that (\ref{eq:am}) will follow by interpolation. 

We first observe that the embedding $L^{2,\frac{1}{2}+\delta} \subset L^1$ and the dispersive estimates $L^1 \rightarrow L^\infty$ give
\begin{align*}
\left\| \int e^{i\tau \phi(\xi-\eta,\eta)} \widehat{f}(\xi-\eta) \widehat{g}(\eta) m(\xi-\eta,\eta)\,d\eta \right\|_{L^2} & 
\lesssim \left\| T_{m}(e^{i\tau b(D)} f, e^{i\tau c(D)} g) \right\|_{L^2} \\
& \lesssim \|e^{i\tau b(D)} f\|_{L^\infty}  \|g\|_{L^2} \\
& \lesssim \tau^{-1/2} \|f\|_{L^{2,1/2+\delta}} \|g\|_{L^2}.
\end{align*}

Moreover, integrating by parts in $\eta$ via the identity $is\partial_\eta \Phi e^{is\Phi} = \partial_\eta e^{is\Phi}$ gives
\begin{align}
& \left\| \int e^{i\tau \Phi(\xi,\eta)} \widehat{f}(\xi-\eta) \widehat{g}(\eta)m(\xi-\eta,\eta) \,d\eta \right\|_{L^2} \\
& \qquad \qquad \qquad \qquad \lesssim \tau^{-1} \left\| \int e^{i\tau\Phi(\xi,\eta)} \partial_{\eta}\left[\partial _\eta\Phi(\xi,\eta)^{-1} \widehat{f}(\xi-\eta) \widehat{g}(\eta)m(\xi-\eta,\eta)\right] \,d\eta \right\|_{L^2}  \nonumber \\
 &\qquad \qquad \qquad \qquad \lesssim \tau^{-3/2} \left[ \|f\|_{L^{2,1/2+\delta}} \|g\|_{L^{2,1}}  + \|f\|_{L^{2,1}} \|g\|_{L^{2,1/2+\delta}} \right], \label{eq:l2}
\end{align}
where we repeat the same arguments as previously.

So let us fix $\tau$ and consider the bilinear operator $U \overset{def}{=}(f,g) \rightarrow  \int e^{i\tau \Phi(\xi,\eta)} \widehat{f}(\xi-\eta) \widehat{g}(\eta) m(\xi-\eta,\eta)\,d\eta$. We have obtained that 
\begin{equation} \|U\|_{L^{2,1/2+\delta} \times L^2 \rightarrow L^2} +  \|U\|_{L^2 \times L^{2,1/2+\delta}  \rightarrow L^2} \lesssim \tau^{-1/2} \label{eq:int1} \end{equation}
and
\begin{equation} \|U(f,g)\|_{L^2} \lesssim \tau^{-3/2} \left[ \|f\|_{L^{2,1/2+\delta}} \|g\|_{L^{2,1}}  + \|f\|_{L^{2,1}} \|g\|_{L^{2,1/2+\delta}} \right]. \label{eq:int2} \end{equation}

We now explain how we can interpolate between these two estimates to obtain the following one:
\begin{equation}
\left\| U(f,g) \right\|_{L^2} \lesssim \tau^{-1+\delta} \|f\|_{L^{2,1/2}} \|g\|_{L^{2,1/2}}, \label{eq:int}
\end{equation}
for any $\delta>0$.
We first consider the collection of dyadic intervals $I_0\overset{def}{=}[-1,1]$ and for $n\geq 1$, $I_n\overset{def}{=}[-2^{n},2^{n-1}] \cup [2^{n-1},2^n]$. On each set $I_n$, the weight $<x>$ is equivalent to $2^n$, so for $n\leq m$ two integers we know from (\ref{eq:int1}) that 
$$ \|U\|_{L^2(I_n) \times L^2(I_m) \rightarrow L^2} \lesssim \tau^{-1/2} 2^{n(1/2+\delta)}$$
and from (\ref{eq:int2}) that 
$$ \|U\|_{L^2(I_n) \times L^2(I_m) \rightarrow L^2} \lesssim \tau^{-3/2} \left[2^{n(1/2+\delta)}2^m + 2^{n} 2^{m(1/2+\delta)}\right] \lesssim \tau^{-3/2} 2^{n(1/2+\delta)} 2^m.$$
Consequently, taking the geometric average with $\delta'>2\delta$, we get 
$$  \|U\|_{L^2(I_n) \times L^2(I_m) \rightarrow L^2} \lesssim \tau^{-1+\delta'} 2^{n(1/2+\delta)} 2^{m(1/2-\delta')} \lesssim \tau^{-1+\delta} 2^{(n+m)(1/2-\delta)}.$$
So we have
\begin{align*}
 \|U(f,g)\|_{L^2} & \lesssim \tau^{-1+\delta'} \sum_{n,m\geq 0} 2^{(n+m)(1/2-\delta)} \|f\|_{L^2(I_n)} \|g\|_{L^2(I_m)} \\
 & \lesssim \tau^{-1+\delta'} \left(\sum_{n,m\geq 0} 2^{-(n+m)\delta}\right) \|f\|_{L^{2,1/2}} \|g\|_{L^{2,1/2}} \\
 & \lesssim  \tau^{-1+\delta'}  \|f\|_{L^{2,1/2}} \|g\|_{L^{2,1/2}}.
\end{align*}
Since $\delta,\delta'$ can be chosen as small as we want with $\delta'>2\delta>0$, $\delta'$ can be chosen arbitrary small, which concludes the proof of (\ref{eq:int}).

Finally from (\ref{eq:int}), we obtain (\ref{eq:am}) for $s=1/2$ by integrating in time for $\tau\in(0,t)$.

\bigskip \noindent \underline{The $L^\infty$ estimate:} We want to prove
\begin{equation}
\|u(t)\|_{L^\infty} \lesssim \left\{ \begin{array}{l} t^{\frac{1}{2} - s + \delta} \|f\|_{L^{2,s}} \|g\|_{L^{2,s}} \quad \mbox{if $0 \leq s \leq 1$} \\ 
t^{-\frac{1}{2}} \|f\|_{L^{2,s}} \|g\|_{L^{2,s}} \quad \mbox{if $s > 1$} \end{array} \right.
\label{eq:amontt}
\end{equation}
The case $s=0$ was stated in Theorem~\ref{main2}. Recall that, writing
$$
u(t) \overset{def}{=} \int_0^t F(t,s) \,ds,
$$
the $L^1 \rightarrow L^\infty$ dispersive estimate gives
\begin{equation}
\label{grebehuppe1}
\left\|F(t,s) \right\|_{L^\infty} \lesssim \frac{1}{\sqrt{t-s}} \|f\|_{L^2} \|g\|_{L^2}.
\end{equation}
Next, integrating by parts via the formula $is\partial_\eta \Phi e^{is\Phi} = \partial_\eta e^{is\Phi}$ gives
\begin{equation*}
\begin{split}
& \int e^{ita(\xi)} e^{is\Phi(\xi,\eta)} \widehat{f}(\xi-\eta) \widehat{g}(\eta) m(\xi-\eta,\eta)\,d\eta \\
 &\qquad \qquad \qquad \qquad = \int e^{ita(\xi)} e^{is\Phi(\xi,\eta)} \frac{1}{i s \partial_\eta \Phi(\xi,\eta)} \partial_\eta \widehat{f}(\xi-\eta) \widehat{g}(\eta) m(\xi-\eta,\eta) \,d\eta \\
& \;\;\;\;\qquad \qquad \qquad \qquad + \int e^{ita(\xi)} e^{is\Phi(\xi,\eta)} \frac{1}{i s \partial_\eta \Phi(\xi,\eta)} \widehat{f}(\xi-\eta) \partial_\eta \widehat{g}(\eta) m(\xi-\eta,\eta) \,d\eta \\
& \;\;\;\;\qquad \qquad \qquad \qquad + \int e^{ita(\xi)} e^{is\Phi(\xi,\eta)} \partial_\eta \left[ \frac{m(\xi-\eta,\eta)}{i s \partial_\eta \Phi(\xi,\eta)} \right]\widehat{f}(\xi-\eta) \widehat{g}(\eta)\,d\eta,
\end{split}
\end{equation*}
which reads in physical space
\begin{equation}
\begin{split}
F(t,s) & = \frac{1}{s} e^{i(t-s) a(D)} T_{\frac{m}{i \partial_\eta \Phi}} \left( e^{itb(D)} x f\,,\,e^{itc(D)} g \right) \\
& \;\;\;\;\;\;\;\; + \frac{1}{s} e^{i(t-s) a(D)} T_{\frac{m}{i \partial_\eta \Phi}} \left( e^{itb(D)} f\,,\,e^{itc(D)} x g \right) \\
& \;\;\;\;\;\;\;\; + \frac{1}{s} e^{i(t-s) a(D)} T_{\partial_\eta \frac{m}{i \partial_\eta \Phi}} \left( e^{itb(D)} f\,,\,e^{itc(D)} g \right) \\
& \overset{def}{=} I + II + III.
\end{split}
\label{eq:split}
\end{equation}
Using the $L^1 \rightarrow L^\infty$ dispersive estimate,
$$
\left\| I \right\|_{L^\infty} \lesssim \frac{1}{s \sqrt{t-s}} \left\| T_{\frac{m}{i \partial_\eta \Phi}} \left( e^{itb(D)} x f\,,\,e^{itc(D)} g \right) \right\|_{L^1} \lesssim \frac{1}{s \sqrt{t-s}} \left\| xf \right\|_{L^2} \left\| g \right\|_{L^2} \lesssim \frac{1}{s \sqrt{t-s}} \|f\|_{L^{2,1}} \|g\|_{L^{2,1}}.
$$
Similar estimates for $II$ and $III$ give
\begin{equation}
\label{grebehuppe2}
\|F(t,s)\|_{L^\infty} \lesssim \frac{1}{s \sqrt{t-s}} \|f\|_{L^{2,1}} \|g\|_{L^{2,1}}.
\end{equation}
Repeating the above argument, but integrating by parts twice via the identity $\frac{1}{is\partial_\eta
\Phi} \partial_\eta e^{is\Phi} = e^{is\Phi}$ 
yields
\begin{equation}
\label{grebehuppe3}
\|F(t,s)\|_{L^\infty} \lesssim \frac{1}{s^2 \sqrt{t-s}} \|f\|_{L^{2,2}} \|g\|_{L^{2,2}}.
\end{equation}
Interpolating between (\ref{grebehuppe1}), (\ref{grebehuppe2}) and~(\ref{grebehuppe3}) gives finally
$$
\|F(t,s)\|_{L^\infty} \lesssim \frac{1}{s^\sigma \sqrt{t-s}} \|f\|_{L^{2,\sigma}} \|g\|_{L^{2,\sigma}}\quad\mbox{for $0 \leq \sigma \leq 2$}.
$$
Integrating this inequality in $s$ (recall that $u(t) = \int_0^t F(t,s)\,ds$) gives the desired estimate.
\end{proof}

\subsection{Decay of solutions: the role of space-time resonances}

We want to consider here the case of a point which would be resonant both in space and in time; we need to combine the two approches which were previously presented.

\begin{thm}
Assume as usual that $m$ is smooth and compactly supported, and that $(H)$ holds. Assume further that the point $p_0\overset{def}{=}(\xi_0,\eta_0)$ is the only point in the 
support of $m$ which is resonant in space and time, or in other words such that $\phi(p_0)=\left( \partial_\xi - \partial_\eta \right) \phi(p_0)=0$.
Moreover, assume that $\phi$ and $\left( \partial_\xi - \partial_\eta \right) \phi$ vanish at order one on their zero sets, and that the two smooth curves $\{ \phi=0 \}$ and 
$\{ \left( \partial_\xi - \partial_\eta \right) \phi =0 \}$ are non tangentially intersecting at $p_0$ with $\partial_\xi \phi(p_0)\neq 0$. \\
Then the solution $u$ of~(\ref{eq:dispersive}) satisfies the following bounds: for $q\in[2,\infty)$ and every $\delta>0$
\begin{itemize}
\item if $s\in[0,\frac{1}{4}]$ then 
$$ \|u(t)\|_{L^q} \lesssim t^{\frac{1}{q}-s\left(\frac{1}{4}+\frac{7}{2q}\right)+\delta} \|f\|_{L^{2,s}} \|g\|_{L^{2,s}}$$
\item if $s\in(\frac{1}{4},1]$ then 
 $$ \|u(t)\|_{L^q} \lesssim t^{-s\left(\frac{1}{4}-\frac{1}{2q}\right)+\delta} \|f\|_{L^{2,s}} \|g\|_{L^{2,s}}.$$
\end{itemize}
\end{thm}

\begin{rem} 
\begin{itemize}
\item For $q=\infty$, the estimates still hold with $BMO$ instead of $L^\infty$.
\item Notice that the assumptions of the theorem imply that, if $\phi$ and $\nabla_\eta \phi$ vanish at order 1 on $\Gamma$ and $\Delta$ respectively,
then at the intersection point of $\Gamma$ and $\Delta$, $\Gamma$ is characteristic along $\xi+\eta$. Fortunately, this turns out not be a problem in the estimates.
\item The technical assumption $\partial_\xi \phi(p_0)\neq 0$ is exactly the same that the one of Theorem \ref{goeland}:  $\Phi_\xi(\xi_0,\eta_0) \neq 0$.
\item In the previous results, for $s=1$, we get that for large $t>>1$
$$ \|u(t)\|_{L^q} \lesssim t^{\frac{1}{2q}-\frac{1}{4}+\delta}\|f\|_{L^{2,1}} \|g\|_{L^{2,1}},$$ for every $\delta>0$.
This estimate is optimal (up to $\delta$ which can be chosen as small as we want) due to the lower bound discussed in Subsubsection \ref{subsub:lower}.
\end{itemize}
\end{rem}

\begin{proof}  The $L^2$ inequalities ($q=2$) have already been proved in Theorem \ref{thm:timereso}. Indeed, from Theorem \ref{thm:timereso} we know that $u(t)$ can be estimated in $L^2$ with a bound $t^{\frac{1-4s}{2}}$ if $s<\frac{1}{4}$ and $t^{\delta}$ for every $\delta>0$ if $s\geq \frac{1}{4}$. Moreover, Theorem \ref{thm:timereso} yields that for every $\delta>0$
$$ \|u(t)\|_{L^\infty} \lesssim t^{\delta} \|f\|_{L^2} \|g\|_{L^2}.$$
So it suffices to check the only one remaining extremal point : $q=\infty$ with $s=1$. We now aim at proving that 
\begin{equation} 
\|u(t)\|_{BMO} \lesssim t^{-\frac{1}{4} + \delta} \|f\|_{L^{2,1}} \|g\|_{L^{2,1}}, \label{eq:ambis}
\end{equation}
which will then imply the desired result by interpolation.

To prove (\ref{eq:ambis}), the main idea is to combine the two previous situations, so let us consider small parameters $\epsilon_1,\epsilon_2\in (t^{-1/2},1)$ and a 
smooth partition of the unity with respect to the domains
$$ \Omega_1\overset{def}{=}\{(\xi,\eta),\ |\phi(\xi,\eta)|>  \epsilon_1 + \frac{1}{2} |\left( \partial_\xi - \partial_\eta \right) \phi(\xi,\eta)| \},$$
$$ \Omega_2 \overset{def}{=} \{(\xi,\eta),\ |\left( \partial_\xi - \partial_\eta \right) \phi(\xi,\eta)|>\epsilon_2+\frac{1}{2}|\phi(\xi,\eta)|\}$$
and
$$ \Omega_3 \overset{def}{=} \{(\xi,\eta),\ |\phi(\xi,\eta)|<2\epsilon_1 \quad \textrm{and} \quad |\left( \partial_\xi - \partial_\eta \right) \phi(\xi,\eta)|< 2 \epsilon_2\}.$$
More precisely, $\Omega_1$ can be thought as a truncated ``cone'' around the curve $|\left( \partial_\xi - \partial_\eta \right) \phi|=0$ and of top $p_0$, similarly 
for $\Omega_2$ around the other curve. This decomposition gives rise from the smooth symbol $m$ to three symbols $m_i$ and we have
$$ u(t)=u_1(t)+u_2(t)+u_3(t)$$
with
$$ \widehat{u_i}(t,\xi) :=  e^{ita(\xi)} \int_0^t \int e^{is\phi(\xi-\eta,\eta)} m_i(\xi-\eta,\eta) \widehat{f}(\xi-\eta) \widehat{g}(\eta) \,d\eta\,ds.$$

\mb
{\underline{Step 1 :}} Estimate of $u_1$ in $BMO$ with $s=1$. \\
We perform the same decomposition, as used previously in the proof of Theorem \ref{main2}, so
$$ u_1(t) = I_t(f,g)-II_t(f,g),$$
with
$$ I_t(f,g) = T_{m_1/\phi}(e^{itb(D)} f, e^{itc(D)}g)$$
and
$$ II_t = e^{ita(D)} T_{m_1/\phi}(f,g).$$
The symbol $m_1$ is of Coifman-Meyer type~\cite{CM} and $\phi$ is smooth and lower-bounded by $\epsilon_1$ so $T_{m_1/\phi}$ is bounded from 
$L^\infty \times L^\infty$ to $BMO$~\cite{CM2} with norm $\lesssim \epsilon_1^{-1}$. Using the dispersive inequalities for the linear evolution groups
\begin{align*}
 \|I_t(f,g)\|_{BMO} & \lesssim \epsilon_1^{-1} \|e^{itb(D)} f\|_{L^\infty} \|e^{itc(D)}g\|_{L^\infty} \lesssim \epsilon_1^{-1} t^{-1} \|f\|_{L^1} \|g\|_{L^1} \\
 & \lesssim \epsilon_1^{-1}t^{-1} \|f\|_{L^{2,1}} \|g\|_{L^{2,1}},
\end{align*}
where we used $L^{2,1} \subset L^1$.
Then we decompose the symbol $m_1$ around $p_0$ for scales $2^{j}$ from $\epsilon_1$ to $1$ (here the scale means the distance in the frequency plane to the point $p_0$, which in $\Omega_1$ is equivalent to $|\phi|$), as follows
$$ m_1 = \sum_{\epsilon_1 \leq 2^j \lesssim 1} m_1 \chi(2^{-j}\phi ),$$
where $\chi$ is a compactly supported and smooth function.
The symbol $m_1\chi(2^{-j}\phi)/\phi$ is of Coifman-Meyer type with a bound $2^{-j}$ so the operator $T_{m_1\chi(2^{-j}\phi)/\phi}$ is bounded from $L^2 \times L^2$ to 
$L^1$ with a bound $2^{-j}$. Since when we evaluate $T_{m_1\chi(2^{-j}\phi)/\phi}(f,g)$, the functions $f$ and $g$ may be assumed supported in frequency on an interval of 
length $2^j$, we deduce from Lemma \ref{lem:tron} that
\begin{align*}
 \|II_t(f,g)\|_{L^\infty} & \lesssim t^{-1/2} \|T_{m_1/\phi}(f,g) \|_{L^1} \\
 & \lesssim  t^{-1/2} \left(\sum_{\epsilon_1\leq 2^j \lesssim 1}  2^{j} 2^{-j}\right) \|f\|_{L^{2,1}} \|g\|_{L^{2,1}} \lesssim t^{-1/2} |\log(\epsilon_1)| \|f\|_{L^{2,1}} \|g\|_{L^{2,1}}.
\end{align*}
So we obtain since $\epsilon_1\in [t^{-1/2},1]$, for every $\delta>0$
\be{eq:u1-1} \|u_1(t)\|_{BMO} \lesssim t^{-1/2}\epsilon_1^{-\delta}  \|f\|_{L^{2,1}} \|g\|_{L^{2,1}}. \ee

\mb
{\underline{Step 2 :}} Estimate of $u_2$ in $L^\infty$ with $s=1$.\\
For $u_2$, we will follow the proof of Theorem \ref{thm:main3}, with the symbol $m_2$ supported on a cone with $|(\xi,\eta)-p_0|\geq \epsilon_2$. 
In our current situation, the symbol $m_2$ satisifes the 
H\"ormander regularity condition (which means $|\partial^\alpha m_2(\xi,\eta)| \lesssim |(\xi,\eta)-p_0|^{-|\alpha|}$) and is supported on $\Omega_2$ which can be considered 
as a cone of top $p_0$. So $\Omega_2$ can be split into different parts distant at $2^j$ from $p_0$ for $\epsilon_2\leq 2^j \lesssim 1$, as follows
$$ m_2 = \sum_{\epsilon_2 \leq 2^j \lesssim 1} m_2 \chi(2^{-j}(\cdot-p_0))$$
where $\chi$ is a smooth and compactly supported function. Following this partition of the unity.
Furthermore, for each of these pieces, $\chi(2^{-j}(\cdot-p_0)$ restricts frequencies to a ball of radius $\sim 2^j$, so it is possible to add projections
$\pi_j$ on $f$ and $g$, where $\pi_j$ projects on intervals of length $\sim 2^j$ which we do not specify.

These considerations lead to the following modification of (\ref{grebehuppe2}):
\begin{subequations}
\begin{align}
\label{aigle1}
\|F(t,s)\|_{L^\infty} & \lesssim \frac{1}{s\sqrt{t-s}} \left[  \sum_{\epsilon_2\leq 2^j \lesssim 1}  \left\| T_{\frac{m_2\chi(2^{-j}(\cdot-p_0))}{\partial_\eta \Phi}}
\right\|_{L^2 \times L^2 \rightarrow L^1} 
\|x f\|_{L^2} \|\pi_j g\|_{L^2} \right.\\
\label{aigle2}
& \quad \quad  + \sum_{\epsilon_2\leq 2^j \lesssim 1} \left\| T_{\frac{m_2\chi(2^{-j}(\cdot-p_0))}{\partial_\eta \Phi}} \right\|_{L^2 \times L^2 \rightarrow L^1} 
\|\pi_j f\|_{L^2} \|xg\|_{L^2} \\
\label{aigle3}
& \quad \quad \left. + \sum_{\epsilon_2\leq 2^j \lesssim 1} \left\| T_{\partial_\eta \frac{m_2\chi(2^{-j}(\cdot-p_0))}{\partial_\eta \Phi}} \right\|_{L^2 \times L^2 \rightarrow L^1}
 \|\pi_j f\|_{L^2} \|\pi_j g\|_{L^2} \right] \\
& \overset{def}{=} \sum_j \left[ I_j + II_j + III_j \right].
\end{align}
\end{subequations}
To bound $I_j$, observe that $2^{j} \frac{m_2\chi(2^{-j}(\cdot-p_0))}{\partial_\eta \Phi}$ is a Coifman-Meyer symbol, thus 
$\left\| T_{\frac{m_2\chi(2^{-j}(\cdot-p_0))}{\partial_\eta \Phi}} \right\|_{L^2 \times L^2 \rightarrow L^1} \lesssim 2^{j}$. Furthermore, by Lemma \ref{lem:tron},
$\| \pi_j g \|_{L^2} \lesssim 2^{j/2} \|g\|_{L^{2,1}}$. Therefore,
$$
I_j \lesssim 2^{-j} \|f\|_{L^{2,1}} \|\pi_j g\|_{L^2} \lesssim 2^{-\frac{j}{2}} \|f\|_{L^{2,1}} \|\pi_j g\|_{L^{2,1}} .
$$
Similarly,
$$
II_j \lesssim 2^{-j} \|f\|_{L^2} \|g\|_{L^{2,1}} \lesssim 2^{-\frac{j}{2}} \|f\|_{L^{2,1}} \|g\|_{L^{2,1}}.
$$
Finally, observe that $2^{2j} \partial_\eta \frac{m_2\chi(2^{-j}(\cdot-p_0))}{\partial_\eta \Phi}$ is a Coifman-Meyer symbol. Applying in addition Lemma~\ref{lem:tron} gives
$$
III_j \lesssim 2^{-j} \|f\|_{L^{2,1}} \|g\|_{L^{2,1}}.
$$
It follows that 
$$ \|F(t,s)\|_{L^\infty} \lesssim \left(\sum_{\epsilon_2\leq 2^j \lesssim 1} 2^{-\frac{j}{2}}+2^{-j} \right) \frac{1}{s \sqrt{t-s}} \|f\|_{L^{2,1}} \|g\|_{L^{2,1}} \lesssim \epsilon_2^{-1} \frac{1}{s \sqrt{t-s}} \|f\|_{L^{2,1}} \|g\|_{L^{2,1}},$$
which means that in (\ref{grebehuppe2}) we get a new extra factor $\epsilon_2^{-1}$. Finally, applying similar arguments as for Theorem \ref{thm:main3}, we conclude that for any $\delta>0$ we have
\begin{equation}
\|u_2(t)\|_{L^\infty} \lesssim \epsilon_2^{-1-\delta} t^{-\frac{1}{2} + \delta} \|f\|_{L^{2,1}} \|g\|_{L^{2,1}}.
\label{eq:u2}
\end{equation}

\mb
{\underline{Step 3 :}} Estimate of $u_3$ in $L^\infty$ with $s=1$. \\
For $u_3$, we know that the symbol $m_3$ is supported on a ball of radius  $\epsilon:=\max\{\epsilon_1,\epsilon_2\}$ around the space-time resonent point $p_0$. 

We follow similar arguments as for Proposition \ref{prop:cas3}, so we split the ball $B(p_0,\epsilon)$ into ``strips" with scale $\phi$ from $0$ to $\epsilon$:
$$ m_3 = \sum_{0< 2^j \lesssim \epsilon} m_3 \chi(2^{-j}\phi)$$
which implies
\begin{align*}
 u_3(t) = \sum_{0<2^k \lesssim \epsilon} T_{m_3^j}(f,g)
\end{align*}
where $T_{m_3^j}$ is the bilinear Fourier multiplier associated to the symbol
$$ m_3^j(\xi,\eta)=e^{ita(\xi+\eta)} m_3(\xi,\eta) \frac{e^{it\phi(\xi,\eta)}-1}{\phi(\xi,\eta)}\chi(2^{-j}\phi(\xi,\eta)).$$
For each scale $2^j$, the symbol $m_3^j$ is bounded by $\max\{t,2^{-j}\}$ and so Lemmas \ref{supercormoran} and \ref{lem:tron} with Remark \ref{rem:car} imply (the functions $f,g$ may be supposed to be frequentially supported on an interval of length $\epsilon$)
$$ \|T_{m_3^j}(f,g)\|_{L^\infty} \lesssim \max\{t,2^{-j}\} 2^j \|f\|_{L^2} \|g\|_{L^2} \lesssim  \max\{t,2^{-j}\} \epsilon \|f\|_{L^{2,1}} \|g\|_{L^{2,1}}.$$
By summing all these inequalities over the scale $2^j$, we get
\begin{align}
\| u_3(t)\|_{L^\infty} & \lesssim \left(t \sum_{2^j \leq t^{-1}} 2^ {j} + \sum_{t^{-1}\leq 2^j \leq \epsilon} 1\right) \epsilon \|f\|_{L^{2,1}} \|g\|_{L^{2,1}} \nonumber \\
 & \lesssim \langle\log(\epsilon t) \rangle \epsilon \|f\|_{L^{2,1}} \|g\|_{L^{2,1}} \nonumber \\
 & \lesssim (\epsilon t)^{\delta} \epsilon \|f\|_{L^{2,1}} \|g\|_{L^{2,1}}, \label{eq:u3}
\end{align}
for every $\delta>0$, since $\epsilon t>1$.

\mb
{\underline{Step 4 :}} End of the proof. \\
Optimizing over $\epsilon_1$ and $\epsilon_2$ leads to
$$ \epsilon_1 = \epsilon_2 = \epsilon_t:=t^{-\frac{1}{4}+\delta}.$$
As required, we have $\epsilon_t\in [t^{-1/2},1]$. So by summing (\ref{eq:u1-1}) and (\ref{eq:u3}) with the estimate for $u_2$, we now have for every small enough $\delta>0$
$$\|u(t)\|_{BMO} \lesssim \left[t^{-\frac{1}{2}}\epsilon_t^{-\delta} + (\epsilon t)^{\delta} \epsilon_t\right] \|f\|_{L^{2,1}} \|g\|_{L^{2,1}}.$$
Since $\epsilon_t\geq t^{-\frac{1}{2}}$, the main term in the previous inequality is the second one so we deduce
for every $\delta>0$
$$  \|u(t)\|_{BMO} \lesssim t^{-\frac{1}{4}+\delta}, $$
which concludes the proof of (\ref{eq:ambis}).
\end{proof}

\appendix

\section{Multilinear estimates}

\label{Ame}

\begin{lem}
\label{cormoran}
Suppose that the symbol $\sigma(\xi,\eta)$ is bounded (i.e. $\|\sigma\|_{L^\infty} \lesssim 1$) and supported on a ball of radius $\epsilon$, say $B(0,\epsilon)$. Then for $q\in[2,\infty]$ and $s< \frac{1}{2}$
$$ \left\| T_\sigma (f,g) \right\|_{L^q} \lesssim \epsilon^{1-\frac{1}{q}+2s} \|f\|_{L^{2,s}} \|g\|_{L^{2,s}}$$
and
$$ \left\| T_\sigma (f,g) \right\|_{L^q} \lesssim \epsilon^{2-\frac{1}{q}} \|f\|_{L^{2,s}} \|g\|_{L^{2,s}}$$
if $s>\frac{1}{2}$.
\end{lem}

\begin{proof} Consider the first claim in the case $s=0$. The lemma is obtained by interpolating between the endpoints $q=2$ and $q=\infty$. If $q=2$, it follows from an application of the Plancherel equality and the Cauchy-Schwarz inequality:
\begin{equation}
\begin{split}
\left\|T_\sigma(f,g)\right\|_{L^2}^2 & = \int \left| \int \sigma(\xi-\eta,\eta) \widehat{f}(\xi-\eta) \widehat{g}(\eta) \,d\eta \right|^2 d\xi  \\
& \leq \int \left(\int |\sigma(\xi-\eta,\eta)|^2 \,d\eta\right) \left( \int \left| \widehat{f}(\xi-\eta) \widehat{g}(\eta) \right|^2 d\eta\right) \,d\xi  \\
& \lesssim \epsilon \|f\|_{L^2}^2 \|g\|_{L^2}^2.
\end{split}
\end{equation}
If $q=\infty$, use Cauchy-Schwarz again to get
\begin{equation}
\begin{split}
\left\|T_\sigma(f,g)\right\|_{L^\infty} & = \left\| \int \int e^{ix(\xi+\eta)} \sigma(\xi,\eta) \widehat{f}(\xi) \widehat{g}(\eta) \,d\eta\,d\xi \right\|_{L^\infty} \lesssim \int \int_{B(0,\epsilon)} \left| \widehat{f}(\xi) \widehat{g}(\eta) \right| \,d\eta\,d\xi \\
& \lesssim \epsilon \left[ \int \int \left| \widehat{f}(\xi) \widehat{g}(\eta) \right|^2 \,d\eta \,d\xi \right]^{1/2} \lesssim \epsilon \|f\|_{L^2} \|g\|_{L^2}.
\end{split}
\end{equation}
Then for $s>0$, we use that the symbol is supported on a ball of radius $\epsilon$, so $f$ (resp. $g$) can be replaced with $\pi_{I}(f)$ (resp. $\pi_J(g)$), corresponding to the frequency-truncation of $f$ on an interval $I$ of length $2\epsilon$. Then we conclude by applying the previous reasoning with $\pi_{I}(f)$ and $\pi_J(g)$ and Lemma \ref{lem:tron}.
\end{proof}

\begin{lem} \label{lem:tron} Assume that $I$ is an interval and consider $\pi_I$ the Fourier multiplier, given by a smooth function supported on $2I$ and equals $1$ on $I$. Then for $q\in[2,\infty]$ and $s< \frac{1}{2}$
$$ \|\pi_I(f)\|_{L^2}  \lesssim |I|^s \|f\|_{L^{s,2}}.$$
\end{lem}

\begin{proof}
The proof relies on the Sobolev embedding as follows
$$ \|\pi_I(f)\|_{L^2} \lesssim |I|^{1/2-1/\sigma} \|\widehat{f}\|_{L^\sigma(2I)} \lesssim |I|^{s} \|\widehat{f}\|_{L^\sigma} \lesssim |I|^{s} \|\widehat{f}\|_{W^{s,2}} \lesssim |I|^s \|f\|_{L^{s,2}},
$$
where the exponent $\sigma$ is given by $\frac{1}{\sigma}=\frac{1}{2}-s$.
\end{proof}

\begin{lem}
\label{supercormoran} Consider first a smooth curve $\Gamma$.
Consider next bounded symbol $\sigma$ ($\|\sigma\|_\infty \lesssim 1$) supported on $\Gamma_\epsilon \cap B(0,M)$, for a positive constant $M$.
Then for any $q\in[2,\infty]$.
\begin{itemize}
\item If the curve $\Gamma$ is nowhere caracteristic, then
$$ \left\|T_\sigma(f,g) \right\|_{L^q} \lesssim \epsilon^{1-\frac{1}{q}} \left\|f\right\|_{L^2} \left\| g \right\|_{L^2}$$
\item If the curve $\Gamma$ has non-vanishing curvature, then
$$ \left\|T_\sigma(f,g) \right\|_{L^q} \lesssim \epsilon^{\frac{3}{4}-\frac{1}{2q}} \left\|f\right\|_{L^2} \left\| g \right\|_{L^2}$$
\item Else, 
$$ \left\|T_\sigma(f,g) \right\|_{L^q} \lesssim \epsilon^{\frac{1}{2}} \left\|f\right\|_{L^2} \left\| g \right\|_{L^2}.$$
\end{itemize}
\end{lem}

\begin{proof} As for Lemma \ref{cormoran}, by interpolation it suffices us to study the two extremal situations: $q=2$ and $q=\infty$.
First for $q=2$, employ the same reasoning as in Lemma \ref{cormoran} (relying on Plancherel equality), since now the support $\Gamma_\epsilon$ has a measure bounded by $\epsilon$ (up to a constant), it comes
\be{eq:L2} \left\|T_\sigma(f,g) \right\|_{L^2} \lesssim \epsilon^{1/2} \|f\|_{L^2} \|g\|_{L^2}. \ee
Let us point out that, this estimate is the easiest situation (when the three exponents are equal to $2$) described by Theorem 1.5 in \cite{BG2}. Moreover this estimate does not depend on geometric properties of the curve $\Gamma$.

Let us now study the case where $q=\infty$.
If the curve $\Gamma$ is nowhere caracteristic, then Proposition 6.2 in \cite{BG2} implies that
$$ \left\|T_\sigma(f,g) \right\|_{L^\infty} \lesssim \epsilon \|f\|_{L^2} \|g\|_{L^2},$$
which by interpolating with (\ref{eq:L2}) proves the desired result.
If the curve $\Gamma$ has a non-vanishing curvature, then Proposition 6.2 in \cite{BG2} yields
$$ \left\|T_\sigma(f,g) \right\|_{L^\infty} \lesssim \epsilon^{3/4} \|f\|_{L^2} \|g\|_{L^2},$$
and we similarly conclude by interpolation.
\end{proof}

\begin{rem} \label{rem:car}
The estimate for a nowhere characteristic curve $\Gamma$ still holds if the curve admits some points which are characteristic only along the variable $\xi+\eta$, 
which means when the tangential vector of the curve at this point is parallel to $(-1,1)$. Indeed the proof of Proposition 6.2 in \cite{BG2} only requires appropriate 
decompositions in the variable $\xi$ and $\eta$ for $f$ and $g$ and do not use specific properties on the third frequency variable $\xi+\eta$.
\end{rem}

\begin{lem}
\label{macareux}
Assume that $\Gamma$ is a non-characteristic curve. Consider a bounded symbol $\sigma$ ($\|\sigma\|_\infty \lesssim 1$) supported on $\Gamma_\epsilon \cap B(0,M)$, for a positive constant $M$. Then
\begin{itemize}
\item If $0\leq s < \frac{1}{4}$, $\displaystyle \left\| T_\sigma (f,g) \right\|_{L^q} \lesssim \epsilon^{1-\frac{1}{q} + \frac{4s}{q}} \|f\|_{L^{2,s}} \|g\|_{L^{2,s}}$.
\item If $s > \frac{1}{4}$, $\displaystyle \left\| T_\sigma (f,g) \right\|_{L^q} \lesssim \epsilon \|f\|_{L^{2,s}} \|g\|_{L^{2,s}}$.
\end{itemize}
\end{lem}

\begin{proof}
It follows the same steps as the previous lemma. The $L^2 \times L^2$ to $L^\infty$ estimate cannot be improved by replacing $L^2$ by $L^{2,s}$, so we simply focus
on $L^2 \times L^2$ to $L^2$ estimates.
Since the curve is assumed to be non characteristic, it follows that 
\begin{equation} \label{eq:d} |\langle T_\sigma(f,g), h\rangle | \lesssim \sum_i \epsilon^ {1/2} \|\widehat{f}\|_{L^2(I^1_i)} \|\widehat{g}\|_{L^2(I^2_i)} \|\widehat{h}\|_{L^2(I^3_i)},\end{equation}
where $(I^k_i)_i$ are collections of almost disjoint intervals of length $\epsilon$ for $k=1,2,3$.
As a consequence, from Cauchy-Schwartz inequality it turns out
$$ \|T_\sigma(f,g)\|_{L^2} \lesssim \epsilon^{1/2} \left( \sup_i \|\widehat{f}\|_{L^2(I^1_i)} \right) \|g\|_{L^2}.$$
Using Sobolev embeding on the whole space $\R$ we get
$$ \|\widehat{f}\|_{L^2(I^1_i)} \lesssim \epsilon^{1/2-1/\sigma} \|\widehat{f}\|_{L^\sigma(I^1_i)} \lesssim \epsilon^{2s} \|\widehat{f}\|_{L^\sigma} \lesssim \epsilon^{2s} \|\widehat{f}\|_{W^{2s,2}} 
$$
with the exponent $\sigma$ given by $\frac{1}{\sigma}=\frac{1}{2}-2s$ (we recall that $s\leq 1/4$). So finally it comes
$$ \|T_\sigma(f,g)\|_{L^2} \lesssim \epsilon^ {1/2+2s} \|f\|_{L^{2,2s}}  \|g\|_{L^2}.$$
By symmetry and then interpolation, we deduce
$$ \|T_\sigma(f,g)\|_{L^2} \lesssim \epsilon^ {1/2+2s} \|f\|_{L^{2,s}}  \|g\|_{L^{2,s}}.$$
\end{proof}

\section{One-dimensional oscillatory integrals}

\label{Aodoi}

\subsection{Main result}

Before stating the main proposition, we need to define the functions 
$$
\mathcal{G}_1(x) \overset{def}{=} \int_x^\infty e^{i\sigma^2} \,d\sigma
\footnote{The function $\mathcal{G}_1$ can obviously be obtained from the Fresnel integrals $S(x) = \int_0^x \sin t^2\,dt$ and $C(x) = \int_0^x \cos t^2\,dt$.
In particular, the constants $C_0$ and $C_{\pm}$ appearing below can be computed via Fresnel integrals to yield $C_0 = (1+i)\sqrt{\frac{\pi}{2}}$ and $C_{\pm}=\frac{1\pm i}{2}
\sqrt{\frac{\pi}{2}}$. See Abramowitz and Stegun~\cite{AS}.}
 \quad \mbox{and} \quad \mathcal{G}_2(x) \overset{def}{=} \int_x^\infty e^{i\sigma^2} \frac{d\sigma}{\sqrt{\sigma-x}}.
$$
Their qualitative behaviour is given by the following lemma.
\begin{lem}
\label{skua}
\noindent (i) $\mathcal{G}_1$ is a smooth function such that
$$
\left\{ \begin{array}{l} 
\mathcal{G}_{1}(x) = - \frac{e^{ix^2}}{2ix} + O\left(\frac{1}{x^2}\right)\quad \mbox{as $x \rightarrow \infty$} \\
\mathcal{G}_{1}(x) = C_0 + O\left(\frac{1}{x}\right)\quad \mbox{as $x \rightarrow - \infty$}
\end{array} \right.
$$
where $C_0$ is the constant $C_0 = \int_{-\infty}^\infty e^{i\sigma^2}\,d\sigma$.

\bigskip
\noindent
(ii) $\mathcal{G}_2$ is a smooth function such that
$$
\mathcal{G}_2(x) = \left\{ \begin{array}{ll} C_{+} e^{ix^2} \sqrt{\frac{2}{x}} + O\left(\frac{1}{|x|^{5/6}}\right) & \mbox{as $x \rightarrow \infty$} \\ 
C_{-} e^{ix^2} \sqrt{\frac{2}{|x|}} + \sqrt{\pi} e^{i\frac{\pi}{4}} e^{-ix^2} \frac{1}{\sqrt{|x|}} + O\left(\frac{1}{|x|^{5/7}}\right) & \mbox{as $x \rightarrow - \infty$} 
\end{array} \right.
$$
where $C_{\pm} = \int_0^\infty e^{\pm i \sigma^2} \,d\sigma$.
\end{lem}
We now state the main result.
\begin{prop} 
\label{fregate}
Let $\chi$ be a smooth, compactly supported function, and let $\zeta$ be a smooth function.

\noindent \bigskip (i) If $\zeta''\geq c>0$, and $\zeta'(\sigma_0)=0$,
$$
\int_0^\infty e^{it\zeta(\sigma)} \chi(\sigma)\,d\sigma = \chi(\sigma_0) \sqrt{\frac{2}{\zeta''(\sigma_0)}} \frac{1}{\sqrt{t}} \mathcal{G}_1(\sqrt{t}\sigma_0) 
+ O_c \left( \frac{1}{t} \right).
$$

\bigskip \noindent (ii) If $|\zeta'|\geq c>0$ does not vanish,
$$\displaystyle \int_0^\infty e^{it \zeta(\sigma)} \chi(\sigma) \frac{d\sigma}{\sqrt{\sigma}} = \frac{\chi(0)}{\sqrt{\zeta'(0)}} e^{it \zeta(0)} \frac{C_0}{\sqrt{t}} + 
O_c\left( \frac{1}{t} \right)$$
(recall that $C_0 = \int_{-\infty}^{+\infty} e^{i\sigma^2}\,d\sigma$).

\bigskip \noindent (iii) If $|\zeta''|\geq c >0$, $\zeta'(\sigma_0)=0$ with $\sigma_0 \geq c$, then
$$\displaystyle \int_0^\infty e^{it \zeta(\sigma)} \chi(\sigma) \frac{d\sigma}{\sqrt{\sigma}} = \frac{\chi(0)}{\sqrt{\zeta'(0)}} e^{i t \zeta(0)} \frac{C_0}{\sqrt{t}} 
+ \sqrt{2\pi} e^{it\zeta(\sigma_0)} e^{i\rho\frac{\pi}{4}} \frac{\chi(\sigma_0)}{\sqrt{\sigma_0 \zeta''(\sigma_0)}} \frac{1}{\sqrt{t}} + O_c \left( \frac{1}{t} \right)$$
(where $\rho = \operatorname{sign}(\zeta''(\sigma_0))$).

\bigskip \noindent (iv) If $\zeta''\geq c >0$, and $\zeta'(\sigma_0)=0$, then
$$
\int_0^\infty e^{it\zeta(\sigma)} \chi(\sigma) \frac{d\sigma}{\sqrt{\sigma}} 
= C(\chi,\zeta) \mathcal{G}_2 (\sqrt{t} \sigma_0) 
+ \left\{ \begin{array}{ll} O_c(t^{-3/4}) & \mbox{if $|\sqrt{t}\sigma_0|<1$} \\ O_c\left( \sqrt{\frac{\sigma_0}{t}} \right) & \mbox{if $|\sqrt{t}\sigma_0|<1$}, \end{array} \right.
$$
where $C(\chi,\zeta)$ is a function of $\chi$ and $\zeta$ (and hence also of $\sigma_0$) which we do not make explicit here.
\end{prop}

\begin{rem} Notice that the statements $(ii)$ and $(iii)$ on the one side, and $(iv)$ on the other side, are complementary: $(ii)$ and $(iii)$ apply when $\zeta'$ vanishes away from zero, or not at all, whereas $(iv)$ is meaningful if the point of vanishing of $\zeta'$ approaches zero.
\end{rem}

\subsection{Proof of Lemma~\ref{skua}}

The assertion $(i)$ is proved by a simple integration by parts, thus we skip it and focus on $(ii)$. After the change of variable of integration $\tau = \sqrt{\sigma - x}$, 
$\mathcal{G}_2$ becomes
$$
\mathcal{G}_2(x) = 2 e^{ix^2} \int_0^\infty e^{i\tau^2 (\tau^2 + 2x)} \,d\tau \overset{def}{=} 2 e^{ix^2} g(x).
$$

\bigskip \noindent \underline{The case $x \rightarrow \infty$} Split then
$$
g(x) = \int_0^R + \int_R^\infty \dots\,d\tau \overset{def}{=} I + II
$$
Start with $I$:
$$
I = \int_0^R e^{i2x \tau^2}\,d\tau + \int_0^R \left[ e^{i\tau^2 (\tau^2 + 2x)} - e^{i2x \tau^2} \right] \,d\tau \overset{def}{=} I_1 +I_2
$$
The term $I_1$ can be written
$$
I_1 = \int_0^\infty e^{i2x \tau^2} \,d\tau - \int_R^\infty e^{i2x \tau^2} \,d\tau = \frac{1}{\sqrt{2x}} \int_0^\infty e^{i\sigma^2}\,d\sigma + O\left(\frac{1}{x R}\right),
$$
where the inequality $\int_R^\infty e^{i2x \tau^2} \,d\tau = O\left(\frac{1}{\alpha R}\right)$ follows by integration by parts.

As for $I_2$, estimate it brutally by
$$
\left| I_2 \right| \lesssim \int_0^R \tau^4 \,d\tau = O(R^5).
$$
Finally, an integration by parts gives
$$
II = \int_0^\infty e^{i\tau^2\left( \tau^2 + 2 \tau x \right)}\,d\tau \lesssim \frac{1}{R^2x}.
$$
Gathering the above gives
$$
g(x) = \sqrt{\frac{1}{2x}} \int_0^\infty e^{i\sigma^2}\,d\sigma + O(R^5) + O\left(\frac{1}{x R^2}\right);
$$
optimizing over $R$ gives finally
$$
g(x) = \sqrt{\frac{1}{2x}} \int_0^\infty e^{i\sigma^2}\,d\sigma + O\left( \frac{1}{x^{5/7}} \right)
$$
which is the desired result.

\bigskip \noindent \underline{The case $x \rightarrow - \infty$} Split then
$$
g(x) = \int_0^{\sqrt{-\frac{x}{2}}} + \int_{\sqrt{-\frac{x}{2}}}^\infty \dots\,d\tau \overset{def}{=} III + IV
$$
Start with $III$. Similarly as for $g$ in the case $x\rightarrow \infty$, split
$$
III = \int_0^R + \int_R^{\sqrt{-\frac{x}{2}}} \dots\,d\tau = III_1 + III_2,
$$
and estimate
$$
III_1 = \frac{1}{\sqrt{2x}} \int_0^\infty e^{i\sigma^2}\,d\sigma + O\left(R^5 + \frac{1}{|x| R}\right) \quad \mbox{and} \quad III_2 = O \left( \frac{1}{R^2x} \right).
$$
Optimizing over $R$ gives
$$
III = \frac{1}{\sqrt{2x}} \int_0^\infty e^{i\sigma^2}\,d\sigma + O\left( \frac{1}{|x|^{5/7}} \right).
$$
Turning now to $IV$, observe that the change of variable $\rho = -\frac{\tau^2}{x}$ gives
$$
IV = \sqrt{-x} \int_{1/2}^\infty e^{ix^2\rho(\rho-2)} \frac{d\rho}{2\sqrt{\rho}} = \frac{\sqrt{\pi}}{2} e^{i\frac{\pi}{4}} e^{-ix^2} \frac{1}{\sqrt{|x|}} + O\left( \frac{1}{|x|} \right).
$$
where the last equality follows by the stationary phase lemma. Putting together our estimates on $III$ and $IV$ gives the desired result.

\subsection{An intermediate result}

The following proposition essentially corresponds to the case of Proposition~\ref{fregate} where $\zeta$ is replaced by either $\sigma$, or $(\sigma-\epsilon)$ 
(in which case $\sigma_0 = \epsilon$).

\begin{prop} \label{manchot} Let $\chi$ be a smooth function.

(i) $\displaystyle \int_\epsilon^\infty e^{it\sigma^2} \chi(\sigma)\,d\sigma = \frac{\chi(0)}{\sqrt{t}} \mathcal{G}_1(\sqrt{t}\epsilon) + O\left(\frac{1}{t}\right)$

\bigskip

(ii) $\displaystyle \int_0^\infty e^{it\sigma} \chi(\sigma) \frac{d\sigma}{\sqrt{\sigma}} = \frac{C_0}{\sqrt{t}} \chi(0) + O\left(\frac{1}{t}\right)$ 
(recall that $C_0 = \int_{-\infty}^{+\infty} e^{i\sigma^2}\,d\sigma$).

\bigskip

(iii) $\displaystyle \int_\epsilon^\infty e^{it\sigma^2} \frac{1}{\sqrt{\sigma-\epsilon}} \chi(\sigma) \,d\sigma = \frac{\chi(0)}{t^{1/4}} \mathcal{G}_2(\sqrt{t}\epsilon) + \left\{ \begin{array}{ll}                                                                                                                                                  
O(t^{-3/4}) & \mbox{if $|\sqrt{t}\epsilon|<1$} \\ O \left( \sqrt{\frac{\epsilon}{t}} \right) & \mbox{if $|\sqrt{t}\epsilon|>1$} \end{array} \right.
$
\end{prop}

\begin{proof} We only prove $(iii)$ since $(i)$ and $(ii)$ are simpler and can be proved by following a similar procedure as $(iii)$

\bigskip \noindent \underline{First reduction for $(iii)$}
The change of variable $\tau = \sqrt{t} \sigma$ gives
$$
\int_\epsilon^\infty e^{it\sigma^2} \frac{1}{\sqrt{\sigma-\epsilon}} \chi(\sigma) \,d\sigma = t^{-1/4} \int_{\sqrt{t}\epsilon}^\infty e^{i\tau^2} \frac{1}{\sqrt{\tau - 
\sqrt{t}\epsilon}} \chi\left( \frac{\tau}{\sqrt{t}} \right) \,d\tau
$$
Thus the proposition will be proved if we can show that
\begin{equation}
\label{phaeton}
\int_{\sqrt{t}\epsilon}^\infty e^{i\tau^2} \frac{1}{\sqrt{\tau - \sqrt{t}\epsilon}} \left[ \chi\left( \frac{\tau}{\sqrt{t}} \right) - \chi(0) \right] \,d\tau = \left\{ \begin{array}{ll}                                                                                                                                                  
O(t^{-1/2}) & \mbox{if $|\sqrt{t}\epsilon|<1$} \\ O \left( \sqrt{\epsilon} t^{-1/4} \right) & \mbox{if $|\sqrt{t}\epsilon|>1$} \end{array} \right.
\end{equation}
Define $\beta$ a smooth, compactly supported function, equal to $1$ on the support of $\chi$. The above term can be written
$$
(\ref{phaeton}) = \chi(0) \int_{\sqrt{t}\epsilon}^\infty e^{i\tau^2} \left[ \beta \left( \frac{\tau}{\sqrt{t}} \right) - 1 \right] \frac{d\tau}{\sqrt{\tau - \sqrt{t}\epsilon}} + 
\int_{\sqrt{t}\epsilon}^\infty e^{i\tau^2} \beta \left( \frac{\tau}{\sqrt{t}} \right) \left[ \chi\left( \frac{\tau}{\sqrt{t}} \right) - \chi(0) \right] 
\frac{d\tau}{\sqrt{\tau - \sqrt{t}\epsilon}}.
$$
Since the first summand is easier to deal with, we focus on the second. Setting $Z(y) \overset{def}{=} \beta(y) \left[ \chi(y) - \chi(0) \right]$, matters reduce thus to proving that
\begin{equation}
\label{guillemot}
\int_{\sqrt{t}\epsilon}^\infty e^{i\tau^2} Z \left( \frac{\tau}{\sqrt{t}} \right) \frac{d\tau}{\sqrt{\tau - \sqrt{t}\epsilon}} \lesssim \left\{ \begin{array}{ll}                                                                                                                                                  
t^{-1/2} & \mbox{if $|\sqrt{t}\epsilon|<1$} \\ \sqrt{\epsilon}t^{-1/4} & \mbox{if $|\sqrt{t}\epsilon|>1$} \end{array} \right.
\end{equation}
where $Z$ is a smooth function vanishing at $0$.

\bigskip \noindent \underline{Proof of~(\ref{guillemot})} Split
$$
(\ref{guillemot}) = \int_{\sqrt{t}\epsilon}^{\sqrt{t}\epsilon + R} + \int_{\sqrt{t}\epsilon + R}^\infty \dots\,d\tau \overset{def}{=} I + II.
$$
The term $I$ is estimated directly, giving
$$
I \leq \int_{\sqrt{t} \epsilon}^{\sqrt{t}\epsilon + R} \left| Z \left( \frac{\tau}{\sqrt{t}} \right) \right| \,\frac{d\tau}{\sqrt{\tau - \sqrt{t}\epsilon}} \lesssim \left\{ \begin{array}{ll} \epsilon \sqrt{R} &
\mbox{if $R<\sqrt{t} |\epsilon|$} \\ t^{-1/2} R^{3/2} & \mbox{if $R>\sqrt{t} |\epsilon|$} \end{array} \right.
$$
The term $II$ is submitted first to an integration by parts using the identity $\frac{1}{2\tau} \partial_\tau e^{i\tau^2} = e^{i\tau^2}$:
\begin{equation*}
\begin{split}
II & = \int_{\sqrt{t}\epsilon + R}^\infty \frac{1}{2\tau} \partial_\tau e^{i\tau^2} Z \left( \frac{\tau}{\sqrt{t}} \right) \frac{d\tau}{\sqrt{\tau - \sqrt{t}\epsilon}} \\
& = \frac{1}{\sqrt{t}} \int_{\sqrt{t}\epsilon + R}^\infty \frac{1}{2} \partial_\tau e^{i\tau^2} \widetilde{Z}
\left( \frac{\tau}{\sqrt{t}} \right) \frac{d\tau}{\sqrt{\tau - \sqrt{t}\epsilon}} \\
& = -\frac{1}{2\sqrt{t}\sqrt{R}} e^{i(\sqrt{t}\epsilon + R)^2} \widetilde{Z} \left( \frac{\sqrt{t}\epsilon + R}{\sqrt{t}} \right) - 
\frac{1}{2t} \int_{\sqrt{t}\epsilon + R}^\infty e^{i\tau^2} \widetilde{Z}'\left( \frac{\tau}{\sqrt{t}} \right) \frac{d\tau}{\sqrt{\tau - \sqrt{t}\epsilon}} \\
& \;\;\;\;\;\;\;\;\;\;\;\;\;\;\;\;\;\;\;\;\;\;\;\;\;\;\;\;\;\;\;\;\;\;\;\;\;\;\;\;\;\;\;\;\;\;\;\; +
\frac{1}{4\sqrt{t}} \int_{\sqrt{t}\epsilon + R}^\infty e^{i\tau^2} \widetilde{Z}\left( \frac{\tau}{\sqrt{t}} \right) \frac{d\tau}{(\tau - \sqrt{t}\epsilon)^{3/2}},
\end{split}
\end{equation*}
where we set $\widetilde{Z}(y) \overset{def}{=} \frac{Z(y)}{y}$. The term $II$ is then estimated directly
\begin{equation*}
\begin{split}
II & \lesssim \frac{1}{2\sqrt{t}\sqrt{R}} \left| \widetilde{Z} \left( \frac{\sqrt{t}\epsilon + R}{\sqrt{t}} \right) \right| + \frac{1}{t} \int_{\sqrt{t}\epsilon + R}^\infty 
\left| \widetilde{Z}'\left( \frac{\tau}{\sqrt{t}} \right) \right| \frac{d\tau}{\sqrt{\tau - \sqrt{t}\epsilon}} \\
&  \;\;\;\;\;\;\;\;\;\;\;\;\;\;\;\;\;\;\;\;\;\;\;\;\;\;\;\;\;\;\;\;\;\;\;\;\;\;\;\;\;\;\;\;\;\;\;\; + \frac{1}{\sqrt{t}} \int_{\sqrt{t}\epsilon + R}^\infty
\left| \widetilde{Z}\left( \frac{\tau}{\sqrt{t}} \right) \right| \frac{d\tau}{(\tau - \sqrt{t}\epsilon)^{3/2}} \\
& \lesssim t^{-1/2} R^{-1/2}.
\end{split}
\end{equation*}
Summing up, we have
$$
I + II \lesssim \left\{ \begin{array}{ll} \epsilon \sqrt{R} + t^{-1/2} R^{-1/2} &
\mbox{if $R<\sqrt{t} |\epsilon|$} \\ t^{-1/2} R^{3/2} + t^{-1/2} R^{-1/2} & \mbox{if $R>\sqrt{t} |\epsilon|$} \end{array} \right.
$$
Optimizing over $R$ (distinguishing of course between the cases $\sqrt{t}|\epsilon|>1$ and $\sqrt{t} |\epsilon|<1$) gives~(\ref{guillemot}).
\end{proof}

\subsection{Proof of Proposition~\ref{fregate}}

We only prove $(iv)$; indeed, the proofs of $(i)$ and $(ii)$ follow closely that of $(iv)$; and $(iii)$ simply requires an additional application of the stationary phase lemma.
The idea is simply to perform a change of variable which reduces matters to the Proposition~\ref{manchot}. We want to estimate
\begin{equation}
\label{petrel}
\int_0^\infty e^{it\zeta(\sigma)} \chi(\sigma)\,\frac{d\sigma}{\sqrt{\sigma}}
\end{equation}
where $\zeta''\geq c>0$ and $\zeta'(\sigma_0)=0$. Set now
$$
y = \Phi(\sigma) \overset{def}{=} \operatorname{sign}(\sigma-\sigma_0) \sqrt{\zeta(\sigma)}
$$
Notice that $\Phi$ is smooth, and that
$$
\Phi^{-1}(0)=\sigma_0 \quad,\quad \Phi'(\sigma_0) = \sqrt{\frac{\zeta''(\sigma_0)}{2}} \quad , \quad (\Phi^{-1})'(0) = \sqrt{\frac{2}{\zeta''(\sigma_0)}} 
$$
Furthermore,
$$
(\Phi^{-1})'(\Phi(0)) = \operatorname{sign}(\Phi(0)-\sigma_0) \frac{2\sqrt{\zeta} (\Phi(0))}{\zeta'(\Phi(0))} \overset{def}{=} C(\zeta)^2
$$
which implies that, $\sqrt{\Phi^{-1}(y)}$ can be written
$$
\sqrt{\Phi^{-1}(y)} = C(\zeta) \sqrt{y} \gamma(y) 
$$
for some smooth, positive function $\gamma$. Performing the change of variable $y = \Phi(\sigma)$ gives
$$
(\ref{petrel}) = \int_{\Phi(0)}^\infty e^{ity^2}  \chi \circ \Phi^{-1}(y) (\Phi^{-1})'(y) C(\zeta)^{-1} \frac{dy}{\sqrt{y} \sqrt{\gamma(y)}}.
$$
Applying Proposition~\ref{manchot} gives the desired result
$$
(\ref{petrel}) = \chi(\sigma_0) \sqrt{\frac{2}{\zeta''(\sigma_0)}} \frac{1}{C(\zeta)} \frac{1}{\gamma(0)} \frac{1}{t^{1/4}} \mathcal{G}_2(\sqrt{t}\epsilon) +
\left\{ \begin{array}{ll} O(t^{-3/4}) & \mbox{if $|\sqrt{t}\epsilon|<1$} \\ O \left( \sqrt{\frac{\epsilon}{t}} \right) & \mbox{if $|\sqrt{t}\epsilon|>1$} \end{array} \right.
$$

\end{document}